\documentclass[11pt,a4paper]{article}

\usepackage{a4wide}
\usepackage{graphicx}
\usepackage{latexsym}
\usepackage{epsfig}
\usepackage{amssymb}
\usepackage{amstext}
\usepackage{amsgen}
\usepackage{amsxtra}
\usepackage{amsgen}
\usepackage{amsthm}

\usepackage{todonotes}

\usepackage{subfigure}

\def\eps{\varepsilon}

\def\d{\,{\rm d}}

\def\RR{\mathbb{R}}

\newtheorem{thm}{Theorem}[section]
\newtheorem{prop}[thm]{Proposition}
\newtheorem{lemma}[thm]{Lemma}
\newtheorem{cor}[thm]{Corollary}
\newtheorem{rem}[thm]{Remark}

\theoremstyle{definition}

\newtheorem{assumption}[thm]{Assumption}

\numberwithin{equation}{section}

\newcommand{\jump}[1]{[\![#1]\!]}
\newcommand{\avg}[1]{\{\!\{#1\}\!\}}

\begin{document}

\title{Flow Characteristics in a Crowded Transport Model }
\author{Martin Burger\thanks{Institut f\"ur Numerische und Angewandte Mathematik, Westf\"alische Wilhelms-Universit\"at (WWU) M\"unster. Einsteinstr. 62, D 48149 M\"unster, Germany. e-mail: martin.burger@wwu.de  } \and Jan-Frederik Pietschmann\thanks{Numerical Analysis and Scientific Computing, Department of Mathematics, TU Darmstadt, Dolivostr. 15, 64293 Darmstadt, Germany. e-mail: pietschmann@mathematik.tu-darmstadt.de  } }
\maketitle

\begin{abstract}
The aim of this paper is to discuss the appropriate modelling of in- and outflow boundary conditions for nonlinear drift-diffusion models for the transport of particles including size exclusion and their effect on the behaviour of solutions. We use a derivation from a microscopic asymmetric exclusion process and its extension to particles entering or leaving on the boundaries. This leads to specific Robin-type boundary conditions for inflow and outflow, respectively. For the stationary equation we prove the existence of solutions in a suitable setup. Moreover, we investigate the flow characteristics for small diffusion, which yields the occurence of a maximal current phase in addition to well-known one-sided boundary layer effects for linear drift-diffusion problems. In a one-dimensional setup we provide rigorous estimates in terms of $\epsilon$, which confirm three different phases. Finally, we derive a numerical approach to solve the problem also in multiple dimensions. This provides further insight and allows for the investigation of more complicated geometric setups. 

\end{abstract}

\section{Introduction}

Transport phenomena of crowded particles and their mathematical modelling have received considerable attention recently, driven by a variety of important applications in biology and social sciences, e.g. transport of ions and macromolecules through channels and nanopores (cf. \cite{Burger2012,Dreyer2013,Dreyer2014,eisenberg2013steric,horng2012pnp}), cargo transport by molecular motors on microtubuli (cf. \cite{ciandrini2014stepping,leduc2012molecular,reese2011crowding}), collective cell migration (cf. \cite{painterhillen,plank2012models,simpson2011models,dyson2014importance}), tumour growth (cf. \cite{jacksonbyrne,stelzer}) or dynamics of human pedestrians (cf. \cite{cividini2013crossing,jelic2012properties,schlakepietschmann}). Such applications naturally lead to questions related to boundary (or interface) conditions restricting in- or outflow of particles, and the resulting characteristics of flow. In ion channels the characteristics are relations between bath concentrations (boundary values) and current, in pedestrian motion one is interested in flow and evacuation properties depending on exit doors, and the movement of cargo between microtubuli respectively delivery to the desired site act as as similar boundary conditions. A variety of computational investigations of such issues have been performed, partly with additional complications such as electrostatic interactions, chemotaxis, or herding. Such simulations can give hints on the flow behaviour, but it becomes difficult to understand the causes and asymptotic regimes for certain effects. 
Therefore we investigate in detail a canonical simple model with in- and outflow boundary conditions in this paper. To this end, we assume that the boundary $\partial\Omega$ of our domain is subdivided into three parts: Inflow $\Gamma$, outflow $\Sigma$ and insulating $\partial\Omega\setminus (\Gamma \cup \Sigma)$ with $\Gamma \cap \Sigma = \emptyset$. Then for $x \in \Omega \subset \RR^n$, $t>0$ and $\rho=\rho(x,t)$  we consider the equation
\begin{align}\label{eq:basic}
 \partial_t \rho + \nabla \cdot j = 0,\quad j =  - D \nabla \rho + \rho (1-\rho) u,
\end{align}
with boundary conditions
\begin{align}\label{eq:bc1}
- j\cdot n &= \alpha(1-\rho),\quad\text{on }\Gamma,\\\label{eq:bc2}
  j\cdot n &= \beta\rho,\quad\text{on }\Sigma,\\
  j\cdot n &= 0,\quad\text{on }\partial\Omega\setminus (\Gamma \cup \Sigma)\label{eq:bc3},
\end{align}
where $u:\RR^n \to \RR^n$ is a given velocity field. The model we use is derived from the paradigmatic asymmetric exclusion process (cf. \cite{chou2011non}), with appropriate modifications to include realistic in- and outflow boundaries. In a simple one-dimensional setup, this model was investigated recently in \cite{Wood2009} including stochastic entrance and exit conditions, exhibiting three different phases of behaviour. We will take a continuum limit of that model and verify that these three phases are still present under the same conditions on parameters and demonstrate how the model generalizes to multiple dimensions and multiple species going in potentially different directions. The (formal) continuum limit naturally leads to the case of nonlinear convection dominating the diffusion, hence the limit of diffusion tending to zero is natural, and indeed the appearing boundary layers are separating the different phases. 

Let us mention that the study of continuum limits of microscopic particle models with size exclusion effects is a very timely research topic. The majority of the rigorous analysis is however carried out for closed systems, i.e. under no-flux boundary conditions or on the whole space, where such systems possess a gradient flow structure that can be well exploited either with transport metrics (cf. \cite{AGS,carrillo2010nonlinear,liero2013gradient}) or with entropy dissipation techniques (cf. \cite{carrillo2001entropy,Burger2010}). The case of non-closed systems has been studied at the continuum level mainly for Dirichlet boundary conditions, where at least the modelling is obvious. In the case of general inflow and outflow conditions the modelling of boundary conditions needs to be adapted to the specific approach for deriving continuum equations (cf. \cite{erban2007reactive}), which seems to have been ignored by most authors in the past. Moreover, the case of non-equilibrium boundary conditions poses additional challenges on the analysis, in particular existence proofs for stationary solutions cannot be carried out anymore by explicit computations or energy minimization arguments. Nonetheless, some arguments can still benefit from the underlying gradient flow structure in the energy, in particular a transformation to dual variables (also called entropy variables) is quite benefitial for existence proofs, since it yields maximum principles that do not hold for the original variables (cf. \cite{Burger2010,Burger2012,Juengel2014}). In this paper we will use similar ideas and extend them from Dirichlet to inflow and outflow boundary conditions. 

This paper is organized as follows: In Section \ref{sec:modelling} we present the model for several species and give more details about the nonlinear boundary conditions. In section \ref{sec:basic} we present existence proofs for a single species. We seperately treat the cases when the velocity field is either a given divergence free vector field or the gradient of a potential function. In Section \ref{sec:asymptotic} we investigate the behaviour for a small diffusion coefficient and compare this to the results presented in \cite{Wood2009}. Finally, in Section \ref{sec:numerics}, we introduce a discontinuous Galerkin scheme and present examples in one and two spatial dimensions.

\section{Modelling}\label{sec:modelling}

Crowding models based on (totally) asymmetric exclusion processes as well as their mean-field continuum limits have gained strong attention recently (cf. e.g. \cite{penington2011building,simpsonmulti} and the references above). The main paradigm is to model jumps of particles on a discrete lattice with jump probabilities consisting of unoriented parts (diffusion) and oriented drifts (transport). The exclusion is incorporated by avoiding jumps to an occupied cell. Using standard continuum limits (rescaling time and space to have grid sizes and typical waiting times converge to zero) as well as simple mean-field closure assumptions, which can also be made rigorous (cf. \cite{giacomin}), one obtains partial differential equations of the form 
\begin{equation}
	\partial_t \rho_i + \nabla \cdot (j_i) = 0, \quad  j_i = - D_i (\rho_0  \nabla \rho_i - \rho_i \nabla \rho_0) + \rho_i \rho_0 u_i, \qquad \text{in } \Omega \times (0,T), \label{basiceqn1} 
\end{equation}
where $x \in \Omega \subset \RR^n$, $t>0$, $\rho_i=\rho_i(x,t)$ is the density of the $i$-th species of particles with diffusion coefficient $D_i$ and velocity field  $u_i:\RR^n \to \RR^n$ $u_i$, $i=1,\ldots,M$. The free-space density $\rho_0$ is given by
\begin{equation}
	\rho_0 = 1- \sum_{i=1}^M \rho_i. 
\end{equation}
In the previously well-investigated case of a potential field $u_i = - D_i \nabla V_i$ for some $V_i:\RR^n \to \RR$ (cf. \cite{Burger2010}), the system can be recast in gradient form 
\begin{equation}
	\partial_t \rho_i = \nabla \cdot ( D_i \rho_0 \rho_i  \nabla (\partial_{\rho_i} E[\rho_1,\ldots,\rho_M]) ), 
\end{equation}
with the entropy functional 
\begin{equation}\label{eq:entropy}
	E[\rho_1,\ldots,\rho_M] = \int_\Omega \left( \sum_{i=1}^M ( \rho_i \log \rho_i - \rho_i V_i) + \rho_0 \log \rho_0 \right)~dx.
\end{equation}
The above differential equations have been studied in detail with potential fields and no-flux boundary conditions, when the system is indeed a gradient flow and stationary solutions can be characterised as minimisers of the entropy at fixed mass (cf. \cite{Burger2010}). In many practical applications different boundary conditions and non-zero flow is of fundamental importance however. In \cite{Burger2012} the case of mixed no-flux and Dirichlet boundary conditions has been studied in a model for charged particles coupled with the Poisson equation. Here we want to focus on in- and outflow boundaries, as recently also used in one-dimensional stochastic models \cite{Wood2009}.

\subsection{Inflow Boundary Conditions}

We assume that particles of the $i$-th species enter the domain $\Omega$ on a subregion $\Gamma_i \subset \partial \Omega$ with rate $\alpha_i > 0$. Without exclusion principle, this would simply mean in the continuum that the normal flux equals $\alpha_i$. Modelling volume exclusion in the discrete setting means that the particle can only enter a grid cell adjacent to the boundary if it is empty. Hence, the probability of entering is modified to $\alpha_i \rho_0$, and we deduce the boundary condition
\begin{equation}\label{eq:bc_in}
	- j_i \cdot n = \alpha_i \rho_0 \qquad \text{on } \Gamma_i.
\end{equation}
Note the negative sign in front of the normal flux since we use the convention of a normal oriented outward. 
The boundary condition can be rewritten as 
$$
D_i \rho_0 \rho_i \nabla \left( \log \frac{\rho_i}{\rho_0} \right) \cdot n =  \rho_0 ( \alpha_i +  \rho_i u_i \cdot n), 
$$
which clarifies the role of the normal velocity at the inflow boundary. The inflow rate can balance the normal velocity only if $u_i \cdot n \leq 0$. On the other hand we will have $\rho_i \leq 1$ and thus, balancing only for $\alpha_i \leq 1$. 

\subsection{Outflow Boundary Conditions}

Outflow boundaries are more straightforward to model, we assume that particles of the $i$-th species leave the domain $\Omega$ on a subregion $\Sigma_i \subset \partial \Omega \setminus \Gamma_i$ with rate $\beta_i$. Thus, 
\begin{equation}\label{eq:bc_out}
	j_i \cdot n = \beta_i \rho_i \qquad \text{on } \Sigma_i. 
\end{equation}
Again we can rewrite the boundary condition in the form
$$
- D_i \rho_0 \rho_i \nabla \left( \log \frac{\rho_i}{\rho_0} \right) \cdot n =  \rho_i ( \beta_i -  \rho_0 u_i \cdot n), 
$$
hence $u_i \cdot n \geq 0$ is needed for balancing. 
Note also that we could easily include no-flux boundaries by simply setting $\beta_i = 0$, but we rather shall consider them explicitly, i.e. we have
\begin{align}\label{eq:bc_no}
 j_i \cdot n = 0\quad \text{on }\; \partial\Omega \setminus \left(\bigcup_{i=1}^M\Gamma_i \cup \bigcup_{i=1}^M \Sigma_i\right). 
\end{align}

\section{Basic Properties}\label{sec:basic}
As detailed in the introduction, we shall from now on restrict our analysis to the case of a single active species ($M=1$). If we further assume that the diffusion coefficient $D$ is normalized to $1$, equation \eqref{basiceqn1} considerably simplifies to
\begin{equation*}
	\partial_t \rho + \nabla \cdot( - \nabla \rho + \rho (1-\rho) u) = 0,
\end{equation*}
with $\rho$ denoting the density of a single species and supplemented with boundary conditions \eqref{eq:bc1}--\eqref{eq:bc3}. We mention that with similar arguments as in \cite{jacksonbyrne,stelzer}, the case  $M=1$ can also be derived from standard continuum fluid mechanical models adding a congestion constraint $\rho_0 = 1 -  \rho$. 
%
Due to the boundary conditions there is obviously no mass conservation in the system. However, there is still a natural balance condition between in- and outflow, i.e., if $\rho$ solves \eqref{eq:basic} then
\begin{equation}
\partial_t \int_\Omega \rho~dx = \int_{\Sigma} \beta \rho ~d\sigma - \int_{\Gamma} \alpha \rho_0~d\sigma,
\end{equation}
where $\alpha$ and $\beta$ denote the in- and outflow rate for $\rho$, respectively. In the stationary case the two integrals need to balance, which implies an interesting coupling in the balance conditions (via $\rho_0$) if $M > 1$. Note also that in an evacuation case, i.e. $\Sigma = \emptyset$, the mass in the system is monotonically decreasing as it is expected. Finally, we state the following assumptions for later use
\begin{assumption}\label{ass:prelim}

(A1) $\Omega \subset \RR^n$, $n=1,2,3$ with boundary $\partial\Omega$ of class $C^2$. 

(A1') In addition to (A1), let $\Gamma$ and $\Sigma$ be such that a weak solution $w$ of the Poisson equation with right-hand side in $L^2(\Omega)$ and Neumann data $\partial_n w = f$ satisfies $w \in H^2(\Omega)$ for any $f$ such that
$$ f|_\Sigma \in H^{1/2}(\Sigma), \quad f|_\Gamma \in H^{1/2}(\Gamma), \quad f|_{\partial \Omega \setminus (\Gamma \cup \Sigma)} \equiv 0. $$  

(A2) $0\le \alpha \le 1$ and $0 \le \beta \le 1$.

(A3) $u\in [W^{1,\infty}(\Omega)]^n$ such that $\nabla\cdot u = 0$, $u\cdot n = -1$ on $\Gamma$, $u\cdot  n = 1$ on $\Sigma$, and $u\cdot  n = 0$ on $\partial \Omega \setminus( \Gamma \cup \Sigma)$.

(A3') $u=\nabla V$, $V\in H^1(\Omega)$.
\end{assumption}

\subsection{Existence of Stationary Solutions}
We shall present two different proofs, one for a given velocity vector field $u$ and a second one where $u = \nabla V$ for some potential $V$, in which case we can employ a transformation to so-called "entropy variables". 
\begin{thm}[incompressible case]\label{thm:incompressible} Let the assumptions (A1), (A1'), (A2) and (A3) hold. Then, the equation
\begin{align*}
 \nabla\cdot (-\nabla \rho + \rho(1-\rho)u) = 0,
\end{align*}
supplemented with the boundary conditions \eqref{eq:bc1}--\eqref{eq:bc3} has at least one solution $u \in H^1(\Omega) \cap L^\infty(\Omega)$ such that  
\begin{equation}
	 \min\{\alpha,1-\beta\} \le \rho(x) \le \max \{ \alpha, 1-\beta \} \label{rhobounds}
\end{equation}
\end{thm}
\begin{proof}
 The proof is based on Schauder's fixed-point theorem. We define the set
 \begin{align*}
  \mathcal{M} = \left\{ \rho \in L^\infty(\Omega)\cap H^1(\Omega)\;|\;\min\{\alpha,1-\beta\} \le \rho \le \max \{ \alpha, 1-\beta \} , \int_\Omega |\nabla \rho|^2dx\leq C \right\}
\end{align*}
with $C= \Vert u\Vert_\infty^2 + 2 \vert \partial \Omega \vert$.
For given $\bar \rho \in \mathcal{M}$, we define the operator $S:\mathcal{M}\rightarrow H^2(\Omega)$ that maps $\bar\rho$ to  the solution of the linearized problem
\begin{align}\label{eq:lin}
 - \nabla\cdot \nabla \rho + (1-2 \bar \rho) \nabla \rho \cdot u = 0,
\end{align}
supplemented with the boundary conditions
\begin{align}\label{eq:lin_bc1}
 \nabla\rho\cdot n &= (\alpha - \rho )(1-\bar\rho),\quad\text{on }\Gamma,\\
 -\nabla\rho\cdot n &= (\beta - (1-\rho) )\bar \rho,\quad\text{on }\Sigma,\\
 \nabla\rho\cdot n &= 0,\quad\text{on }\partial\Omega\setminus (\Gamma \cup \Sigma)\label{eq:lin_bc3}.
\end{align}
Note that we linearized the boundary conditions differently on $\Gamma$ and $\Sigma$, this will be crucial later on. Standard theory for linear elliptic equations (and our assumption on the regularity of the boundary), cf. \cite{Grisvard1985,Ladyzhenskaya1968}, ensures a maximum principle and existence of a weak solution,  subsequently the existence of a solution $\rho\in H^2(\Omega)$ since the prerequisites of (A1') are satisfied. In order to apply Schauder's fixed-point theorem, we have to prove that $S$ is self-mapping from $\mathcal{M}$ to $\mathcal{M}$, continuous and compact.

{\bf Self-mapping:}
Equation \eqref{eq:lin} satisfies a maximum principle with vanishing normal derivative on $\partial \Omega \setminus (\Gamma \cup \Sigma)$, and thus (by Hopf's maximum principle) $\rho$ attains its maximum on $\Gamma \cup \Sigma$. We have to distinguish the following cases: 
\begin{itemize}
 \item $\rho$ attains its maximum on $\Gamma$ and thus $\nabla\rho\cdot n\ge 0$. Since by assumption $(1-\bar\rho) \ge 0$, this implies $\alpha + \rho u\cdot  n \ge 0$. As $u\cdot  n  = - 1 $ on $\Gamma$ we conclude $\rho \le \alpha$.
 \item If $\rho$ attains its maximum on $\Sigma$ we conclude, since $u\cdot n = 1$, $\rho\le 1-\beta$.
\end{itemize}
If $\rho$ attains its minimum on the boundary we use the same arguments to conclude $\alpha \le \rho$ and $ 1-\beta \le \rho$. Finally, the $L^2$-bound on $\nabla \rho$ follows by using the weak formulation with test function $\rho$ and applying the bounds $0\leq \rho \leq 1$ and $0 \leq \bar \rho \leq 1$.

{\bf Continuity:} To show continuity of $S$ we take a sequence $\bar\rho_k$ in $L^\infty(\Omega)\cap H^1(\Omega)$ such that $\bar\rho_k \rightarrow \bar\rho$. We have to show that the sequence $\rho_k = S(\bar\rho_k)$ converges to some $\rho$ and that $\rho = S(\bar\rho)$. Since $\rho_k \in H^2(\Omega)$ we know that there exists $\tilde\rho$ such a subsequence that $\bar\rho_{k_j} \rightharpoonup \tilde{\rho}$ weakly in $H^2(\Omega)$.
Thus
\begin{align*}
 \int_\Omega \nabla \rho_{k_j}\cdot \nabla\phi\d x - \int_\Omega \rho(1-\rho_{k_j})u\cdot \nabla \phi\d x \rightarrow \int_\Omega \nabla \tilde{\rho}\cdot\nabla\phi\d x - \int_\Omega \rho(1-\tilde\rho)u\cdot \nabla \phi\d x,
\end{align*}
i.e. $\tilde\rho$ solves \eqref{eq:lin}. The continuity of the trace operator allows us to pass to the limit in the boundary conditions \eqref{eq:lin_bc1}--\eqref{eq:lin_bc3} as well. The maximum principle discussed above implies that this solution is unique and the uniqueness of limits therefore yields $\tilde\rho = \rho$.

{\bf Compactness:} The compactness of the operator $S$ follows from the fact that the embedding $H^2(\Omega)\hookrightarrow L^\infty(\Omega)\cap H^1(\Omega)$ is compact for $n\le 3$. This completes the proof.
\end{proof}
Next we treat the potential case $u = \nabla V$, where we obtain the following
\begin{thm}[potential case]\label{thm:M1_potential} Let the assumptions (A1), (A2) and (A3') hold. Then, the equation
\begin{align}\label{eq:M1_potential}
 \nabla\cdot (-\nabla\rho + \rho(1-\rho)\nabla V) = 0,\quad x\in\Omega,
\end{align}
supplemented with the boundary conditions \eqref{eq:bc1}--\eqref{eq:bc3} has at least one solution $u \in H^1(\Omega) \cap L^\infty(\Omega)$ such that $0 \le \rho \le 1$.
\end{thm}
\begin{proof}
Our proof is based on an approximation procedure, applied to the equation in \emph{entropy variables} which are defined as the variation of the entropy functional with respect to the density $\rho$. Since we are in the case $M=1$, the entropy functional \eqref{eq:entropy} reduces to
\begin{align*}
 E[\rho] = \int_\Omega \rho\ln(\rho) - \rho V + (1-\rho)\ln (1-\rho).
\end{align*}
Then, we introduce the entropy variable as
\begin{align*}
 \psi := \partial_{\rho} E[\rho] = \log \rho - \log \rho_0 - V.
\end{align*}
Using elementary calculations, we can express the original density $\rho$ as
\begin{align*}
 \rho = \frac{e^{\psi+V}}{1+e^{\psi+V}},\;\text{and}\;\rho_0 = \frac{1}{1+e^{\psi+V}}.
\end{align*}
Applying this transformation to \eqref{eq:M1_potential}, \eqref{eq:bc1}--\eqref{eq:bc3} yields the nonlinear equation
\begin{align}\label{eq:trans_entropy}
 - \nabla\cdot\left(\frac{e^{\psi+V}}{\left(1+e^{\psi+V}\right)^2}\nabla \psi \right) = 0,
\end{align}
supplemented with the boundary conditions
\begin{align}\label{eq:bc_in_ent_1}
\frac{e^{ \psi+V}}{(1+e^{ \psi+V})^2}\nabla \psi\cdot  n &= \alpha \frac{1}{1+e^{ \psi+V}},\quad\text{on }\Gamma,\\\label{eq:bc_out_ent_1}
-\frac{e^{ \psi+V}}{(1+e^{ \psi+V})^2}\nabla \psi\cdot  n &= \beta\frac{e^{ \psi+V}}{1+e^{ \psi+V}},\quad\text{on }\Sigma,\\\label{eq:bc_no_ent_1}
\frac{e^{ \psi+V}}{(1+e^{ \psi+V})^2}\nabla \psi\cdot  n &= 0,\quad\text{on }\partial\Omega\setminus (\Gamma \cup \Sigma).
\end{align}
We will now apply an approximation procedure to this equation and proceed in several steps.

{\bf Existence for an auxiliary problem}: To simplify our notation we introduce the function $A(\psi,V) := \frac{e^{\psi+V}}{\left(1+e^{\psi+V}\right)^2}$ and for $\delta > 0$ we consider the problem
\begin{align}\label{eq:approx1}
 -\nabla\cdot( A(\psi^\delta,V) \nabla\psi^\delta) + \delta\psi^\delta = 0.
\end{align}
To prove existence of \eqref{eq:approx1}, we use a fixed-point argument. For given $\tilde\psi \in L^2(\Omega)$ we define $ \tilde A_\delta(x) = A(\tilde\psi(x),V(x)) + \delta$ which yields the linear equation
\begin{align}\label{eq:ent_aux}
 -\nabla\cdot( \tilde A_\delta \nabla\tilde{\psi}^\delta) + \delta\tilde{\psi}^\delta = 0,
\end{align}
subject to the nonlinear boundary conditions
\begin{align}\label{eq:ent_aux_bc}
 \tilde A_\delta\nabla\tilde{\psi}^\delta\cdot n = \left\{\begin{array}{ll}
                              \alpha\frac{1}{1+e^{\tilde{\psi}^\delta+V}}, &\text{on }\Gamma,\\
                              -\beta\frac{e^{\tilde{\psi}^\delta+V}}{1+e^{\tilde{\psi}^\delta+V}}, &\text{on }\Sigma,\\
                              0, & \text{on }\partial\Omega\setminus (\Gamma \cup \Sigma).
                             \end{array}\right.
\end{align}
The corresponding weak formulation, for $\varphi \in H^1(\Omega)$, is given by
\begin{align}\label{eq:weak_linear}
 0=\int_\Omega \left( \tilde A_\delta \nabla\tilde{\psi}^\delta\cdot\nabla\varphi + \delta \tilde{\psi}^\delta\varphi\right)\d x - \alpha\int_\Gamma \frac{1}{1+e^{\tilde{\psi}^\delta+V}}\varphi\d \sigma + \beta\int_\Sigma\frac{e^{\tilde{\psi}^\delta+V}}{1+e^{\tilde{\psi}^\delta+V}}\varphi\d\sigma.
\end{align}
This is the Euler-Lagrange equation to the nonlinear minimisation problem for the energy functional
\begin{align*}
 E(\tilde{\psi}^\delta) = \frac12 \int_\Omega \left(\tilde A_\delta|\nabla\tilde{\psi}^\delta|^2 + \delta |\tilde{\psi}^\delta|^2\right)\d x - \alpha\int_\Gamma F(\tilde{\psi}^\delta,V)\d\sigma + \beta\int_\Sigma G(\tilde{\psi}^\delta,V)\d\sigma,
\end{align*}
where $F$ and $G$ are chosen such that 
\begin{align*}
 \partial_\psi F(\psi,V) = \frac{1}{1+e^{\psi+V}},\quad \partial_\psi G(\psi,V) = \frac{e^{\psi+V}}{1+e^{\psi+V}}.
\end{align*}
Note that $F$, $G$ are convex, since their second derivatives are non-negative. Furthermore, due to 
\begin{align*}
 \tilde A_\delta(x) = \frac{e^{\tilde{\psi} + V}}{(1+e^{\tilde{\psi} + V})^2} + \delta =\frac{1}{2(1+\cosh(\tilde{\psi} + V))} + \delta,
\end{align*}
we have that $\tilde A_\delta(x) \in L^\infty(\Omega)$, uniformly with $\delta \le \tilde A_\delta(x) \le \delta + 1/4$. Thus $E(\tilde{\psi}^\delta)$ is coercive with respect to the $H^1$ norm and due to its convexity we conclude (cf. \cite[Section 8.2, theorems 2 and 3]{Evans98}) the existence of a unique minimiser $\tilde{\psi}^\delta \in H^1(\Omega)$, which is by definition a weak solution to \eqref{eq:ent_aux}. Furthermore, since $G(\tilde{\psi}^\delta,V) \ge 0$ and $F(\tilde{\psi}^\delta,V) < \infty$, we infer the $L^2$ a-priori estimate  
\begin{align}\label{eq:schauder_apriori}
 \int_\Omega |\tilde{\psi}^\delta|^2 \d x \le \tilde C + \alpha\int_\Gamma F(\tilde{\psi}^\delta,V)\d\sigma \le C_{\mathcal{M}}.
\end{align}
Here the constant $C_{\mathcal{M}}$ depends on the geometry, $\alpha$, $\beta$, and $\delta$. This result allows us to define the nonlinear operator $\tilde K:L^2(\Omega)\to H^1(\Omega)$ mapping $\tilde\psi$ to the solution of \eqref{eq:ent_aux}--\eqref{eq:ent_aux_bc}. Our aim is to apply Schauder's fixed point theorem in the set 
\begin{align*}
 \mathcal{M} := \{ \psi \in L^2(\Omega)\,|\, \|\psi\|_{L^2(\Omega)}^2 \le C_{\mathcal M} \}.
\end{align*}
To this end we denote by $I_{H^1(\Omega)\hookrightarrow L^2(\Omega)}$ the compact embedding of $H^1$ into $L^2$ and define the operator $K:\mathcal M \to \mathcal M$ by $K = I_{H^1(\Omega)\hookrightarrow L^2(\Omega)}\circ \tilde K$. Since the a-priori estimate \eqref{eq:schauder_apriori} implies that $K$ is self-mapping, it remains to show its continuity. We consider a sequence $\tilde\psi_n$ that converges to $\tilde \psi$ in $L^2(\Omega)$. This yields a sequence $\tilde{\psi}_n^\delta \in H^1(\Omega)$ having a weak limit. Since $\tilde A_\delta$ is uniformly bounded in $L^\infty(\Omega)$, an application of Lebesgues Theorem yields $\tilde A_\delta(\tilde \psi_n) \to \tilde A_\delta(\tilde \psi)$ in $L^2(\Omega)$ and thus we can pass to the limit in the first integral of the weak formulation \eqref{eq:weak_linear}, i.e.
\begin{align*}
 \int_\Omega \tilde A_\delta(\tilde\psi_n) \nabla \tilde \psi^\delta_n \cdot\nabla\varphi\d x \to  \int_\Omega \tilde A_\delta(\tilde \psi) \nabla\tilde \psi^\delta \cdot\nabla\varphi\d x.
\end{align*}
Uniqueness of the weak solution to \eqref{eq:weak_linear} (due to the convexity of the Energy $E$) thus implies $K(\tilde \psi_n) \to K(\tilde\psi)$ in $L^2(\Omega)$. Thus Schauder's fixed point theorem yields the existence of a solution $\psi^\delta$ to the auxiliary problem \eqref{eq:ent_aux}--\eqref{eq:ent_aux_bc}. 

{\bf Limit $\mathbf{\delta \to 0}$}: To this end, we define
\begin{align*}
 \rho^\delta = \frac{e^{\psi^\delta+V}}{1+e^{\psi^\delta+V}}. 
\end{align*}
Then $\rho^\delta\in H^1(\Omega)$ and satisfies the equation
\begin{align}\label{eq:ent_approx}
 \nabla\cdot\left(-\nabla\rho^\delta + \rho^\delta(1-\rho^\delta)\nabla V \right) - \delta\Delta \psi^\delta + \delta\psi^\delta = 0,
\end{align}
with the boundary conditions
\begin{align}
 \left(-\nabla\rho^\delta + \rho^\delta(1-\rho^\delta)\nabla V \right)\cdot  n = \left\{\begin{array}{ll}
                              -\alpha(1-\rho^\delta), &\text{on }\Gamma,\\
                              \beta\rho^\delta, &\text{on }\Sigma,\\
                              0, & \text{on }\partial\Omega\setminus (\Gamma \cup \Sigma).
                             \end{array}\right.
\end{align}
Again, we consider the weak form given by 
\begin{align}\label{eq:rho_delta_weak}
0&= \int_\Omega \left(\nabla\rho^\delta - \rho^\delta(1-\rho^\delta)\nabla V \right)\cdot\nabla\varphi +\delta\nabla\psi^\delta\cdot\nabla\varphi+ \delta\psi^\delta\varphi\d x \\
& - \alpha\int_\Gamma (1-\rho^\delta)\varphi\d\sigma + \beta\int_\Sigma \rho^\delta\varphi\d\sigma, \quad \varphi \in H^1(\Omega).
\end{align}
Our aim is to derive a-priori estimates on $\rho^\delta$ by choosing the test function $\varphi=\psi^\delta$. We have
\begin{align}\label{eq:estimate}
 0&=\int_\Omega (\nabla\rho^\delta\cdot \nabla\psi^\delta - \rho^\delta (1-\rho^\delta)\nabla V\cdot \nabla \psi^\delta)\d x + \delta\int_\Omega |\nabla\psi^\delta|^2 + (\psi^\delta)^2\d x\\\nonumber
 &- \alpha\int_\Gamma (1-\rho^\delta)\psi^\delta\d\sigma + \beta\int_\Sigma \rho^\delta\psi^\delta\d\sigma.
\end{align}
We estimate each term separately, noting that $\nabla\psi^\delta = \frac{1}{\rho^\delta(1-\rho^\delta)}\nabla\rho^\delta - \nabla V$. For the first term we have
\begin{align*}
& \int_\Omega \frac{|\nabla\rho^\delta|^2}{\rho^\delta(1-\rho^\delta)}\d x - 2 \int_\Omega \nabla V\cdot\nabla \rho^\delta\d x + \int_\Omega \rho^\delta(1-\rho^\delta)|\nabla V|^2\d x\\
&\ge \frac{1}{2}\int_\Omega \frac{|\nabla\rho^\delta|^2}{\rho^\delta(1-\rho^\delta)}\d x - \int_\Omega \rho^\delta(1-\rho^\delta)|\nabla V|^2\d x\\
&\ge 2\int_\Omega |\nabla\rho^\delta|^2\d x - \frac{1}{4} \int_\Omega |\nabla V|^2\d x,
\end{align*}
where we used Cauchy's inequality to estimate the mixed term and the fact that $\rho^\delta(1-\rho^\delta) \le 1/4$. For the second term we estimate
\begin{align*}
 &-\alpha\int_\Gamma (1-\rho^\delta)\psi^\delta \d\sigma = \alpha \int_\Gamma  \left((1-\rho^\delta)\log \frac{1-\rho^\delta}{\rho^\delta} + 2\rho^\delta - 1 \right)\d \sigma
+\alpha\int_\Gamma \left(-(1-\rho^\delta)V + 1 - 2\rho^\delta\right)\d\sigma.
\end{align*}
The first term in this equation is a Kullback-Leibler distance and thus non-negative. As $V\in H^1(\Omega)$, the trace theorem yields $\left. V\right|_{\partial\Omega} \in L^2(\partial\Omega)$ and since, by definition $0\le \rho^\delta \le 1$, the second term is bounded. For the third term of \eqref{eq:estimate} we write
\begin{align*}
\beta\int_\Sigma \rho^\delta\psi^\delta\d\sigma = \beta \int_\Sigma \left(\rho^\delta \log \frac{\rho^\delta}{1-\rho^\delta} + 2\rho^\delta - 1 \right)\d \sigma - \beta\int_\Sigma \left(-\rho^\delta V - 1 + 2\rho^\delta\right)\d \sigma.
\end{align*}
By the same arguments as above, we conclude that the first term is non-positive while the second one is bounded. 
Summarizing, we obtain
\begin{align*}
 \int_\Omega |\nabla\rho^\delta|^2\d x \le \frac{1}{8} \int_\Omega |\nabla V|^2\d x + \alpha\int_\Gamma \left(-(1-\rho^\delta)V + 1 - 2\rho^\delta\right)\d\sigma - \beta\int_\Sigma \left(-\rho^\delta V - 1 + 2\rho^\delta\right)\d \sigma.
\end{align*}
These estimates yield a a-priori bound for $\rho^\delta$ in $H^1(\Omega)$. Due to $0\le \rho^\delta \le 1$, this allows us to pass to the limit in the weak formulation \eqref{eq:rho_delta_weak}. In particular we have, by passing to subsequences if necessary, 
\begin{align*}
\int_\Omega \nabla\rho^\delta\cdot\nabla\varphi\d x &\to  \int_\Omega \nabla\rho\cdot\nabla\varphi\d x &\text{ since } \nabla\rho^\delta\rightharpoonup\nabla \rho\text{ in }L^2(\Omega),\\
\int_\Omega \rho^\delta(1-\rho^\delta)\nabla V\cdot\nabla\varphi\d x &\to \int_\Omega \rho(1-\rho)\nabla V\cdot\nabla\varphi\d x &\text{ since } \rho^\delta \to \rho \text{ in }L^2(\Omega),\\
\delta\int_\Omega \nabla\psi^\delta\cdot\nabla\varphi + \psi^\delta\varphi\d x &\to 0 &\text{ since } \psi^\delta,\,\nabla\psi^\delta\in L^2(\Omega),\\
- \alpha\int_\Gamma (1-\rho^\delta)\varphi\d\sigma + \beta\int_\Sigma \rho^\delta\varphi\d\sigma &\to - \alpha\int_\Gamma (1-\rho)\varphi\d\sigma + \beta\int_\Sigma \rho\varphi\d\sigma &\text{ since } \rho^\delta \to \rho \text{ in }L^2(\partial\Omega).
\end{align*}
\end{proof}

Note in the potential case we can only conclude that the density $\rho$ takes values between zero and one, for a stronger result depending on the in- and outflow parameters we need to return to the assumptions for incompressible velocity fields:

\begin{cor}\label{cor:M1_max} Let the assumptions of theorem \ref{thm:M1_potential} and additionally $\Delta V = 0$, $\partial_n V = -1$ on $\Gamma$, $\partial_n V = 1$ on $\Sigma$, and
$\partial_n V=0$ on $\partial \Omega \setminus (\Gamma \cup \Sigma)$ hold. Then, the solution $\rho$ to \eqref{eq:M1_potential} supplemented with \eqref{eq:bc1}--\eqref{eq:bc3} satisfies the bounds
\begin{equation}
	 \min\{\alpha,1-\beta\} \le \rho(x) \le \max \{ \alpha, 1-\beta \}
\end{equation}
\end{cor}
\begin{proof}
Since for $\Delta V = 0$ the equation \eqref{eq:M1_potential} fulfills a maximum principle, the assertion follows by similar arguments as in the proof of Theorem \ref{thm:incompressible}.
\end{proof}

\begin{rem} Note that testing the weak form \eqref{eq:rho_delta_weak} with the entropy variable $\psi^\delta$ yields to estimates analogous to those that are obtained by the entropy dissipation method in the time dependent case. In fact, if equation \eqref{eq:ent_approx} would feature the additional term $\partial_t \rho$ (parabolic case), then differentiating the entropy functional with respect to time would yield 
\begin{align*}
 \partial_t E(\rho^\delta) = \int_\Omega \left[\nabla\cdot((-\nabla\rho^\delta + \rho^\delta(1-\rho^\delta)\nabla V - \delta \nabla \psi^\delta) + \delta\psi^\delta\right] (\ln \rho^\delta - \ln (1-\rho^\delta) - V)\;dx.
\end{align*}
Recalling the definition $\psi^\delta = (\ln \rho^\delta - \ln (1-\rho^\delta ) - V)$ and after integration by parts one obtains
\begin{align}\label{eq:ent_diss_approx}
 \partial_t E(\rho^\delta) &= -\int_\Omega  \left(-\nabla\rho^\delta + \rho^\delta(1-\rho^\delta)\nabla V   \right)\cdot \nabla \psi^\delta  - \delta |\nabla \psi^\delta|^2 + \delta |\psi^\delta|^2\;dx\\
            &-\alpha \int_\Gamma (1-\rho^\delta)\psi^\delta\;d\sigma  + \beta\int_\Sigma \rho^\delta\;d\sigma.\nonumber
\end{align}
This means that in the entropy dissipation, we obtain the same terms as in equation \eqref{eq:estimate}. While in the stationary case, their sum is zero, we can conclude boundedness in the parabolic case by integrating \eqref{eq:ent_diss_approx} with respect to time to conclude 
\begin{align*}
E(\rho^\delta) &+ \int_0^T \int_\Omega  \left(-\nabla\rho^\delta + \rho^\delta(1-\rho^\delta)\nabla V   \right)\cdot \nabla \psi^\delta  - \delta |\nabla \psi^\delta|^2 + \delta |\psi^\delta|^2\;dx\,dt\\
            &-\alpha \int_0^T\int_\Gamma (1-\rho^\delta)\psi^\delta\;d\sigma\,dt  + \beta\int_0^T\int_\Sigma \rho^\delta\;d\sigma\,dt  \le E(\rho_0) \le C,
\end{align*}
where $\rho_0$ denotes the initial datum.

\end{rem}

\begin{rem} We finally mention that for convenience we used a diffusion coefficient equal to one, but all results of this section remain true for an arbitrary value $D  > 0$ and even for regular spatially varying coefficients. This is important for the following section, where study the natural case of a small diffusion coefficient.
\end{rem}

\section{Asymptotic Unidirectional Flow Characteristics}\label{sec:asymptotic}

In the following we discuss in detail the flow properties of the single species model for small diffusion $D = \eps \ll 1$, in particular for the stationary solution $\rho \in H^1(\Omega) \times L^\infty(\Omega)$ of 
\begin{equation}
	 \nabla \cdot( -\epsilon \nabla \rho + \rho (1-\rho) u) = 0, \qquad \text{in } \Omega \label{unifloweq}
\end{equation}
with boundary conditions \eqref{eq:bc1}--\eqref{eq:bc3}.
We are interested in the asymptotic behaviour as $\epsilon \downarrow 0$, in particular the boundary layers and the asymptotic flow patterns, which we expect to be characterised by three different phases as in \cite{Wood2009}:
\begin{itemize}

\item An {\em influx-limited} phase with an asymptotically low density corresponding to a density of outgoing particles on $\Sigma$.

\item An {\em outflux-limited} phase with an asymptotically high density corresponding to a density of incoming particles on $\Gamma$, with a boundary layer on $\Sigma$ created by lower outflux rates.

\item A {\em maximal current } phase with asymptotic density $\frac{1}2$ and boundary layers both at $\Sigma$ and $\Gamma$, which occurs at high in- and outflow rates.
\end{itemize}
First we note that a direct conclusion from the maximum principle \eqref{rhobounds} is the non-appearance of a maximal current phase if 
\begin{align*}
\frac{1}2 \notin \left[ \min\{\alpha, (1-\beta) \}  , \max\{\alpha,(1-\beta) \}  \right]. 
\end{align*}
In this case the maximum principle implies that the densities are bounded away from $\frac{1}2$ uniformly in $\epsilon$. 


\subsection{Characterization of Phases in the One-Dimensional Flow}

We now turn to the one-dimensional case with constant velocity, where we can use a scaling of space and flow such that $\Omega=[0,1]$, 
$u \equiv 1$, $\Gamma =\{0\}$, and $\Sigma=\{1\}$. This setting corresponds exactly to the continuum limit of the setting in \cite{Wood2009}, and we will rigorously show that indeed the same behaviour as in the TASEP with stochastic entrance and exit conditions - with transitions between the phases at exactly the same parameter values - appears for the continuum limit.
For convenience we restate the one-dimensional version of \eqref{unifloweq}, \eqref{eq:bc1}--\eqref{eq:bc3} as
\begin{equation}
- \epsilon \partial_{xx} \rho + \partial_x ( \rho (1-\rho) ) =  0 \qquad \text{in } (0,1), \label{uniflow1deq}
\end{equation}
with boundary conditions
\begin{align}
	&\epsilon \partial_x \rho =  (1- \rho) ( \rho - \alpha) & \text{at } x=0, \\
	&\epsilon \partial_x \rho =  \rho (1 - \rho - \beta) & \text{at } x=1. \label{uniflow1dbc2}
\end{align}

A first result particular for the one-dimensional case is the uniqueness of a solution:
\begin{prop}
There exists exactly one weak solution $\rho \in H^1(\Omega)$ of \eqref {uniflow1deq}-\eqref{uniflow1dbc2}.
\end{prop}
\begin{proof}
Let $\rho_1$ and $\rho_2$ be two solutions, then $w=\rho_1-\rho_2$ satisfies
$$  - \epsilon \partial_{xx} w + \partial_x ((1-\rho_1-\rho_2) w ) = 0 $$
with boundary conditions
\begin{align*}
	&-\epsilon \partial_x w + (1-\rho_1-\rho_2) w =  - \alpha w & \text{at } x=0, \\
	&-\epsilon \partial_x w + (1-\rho_1-\rho_2) w =  \beta w & \text{at } x=1. 
\end{align*}
Now let $V \in H^2([0,1])$ be such that $-\epsilon \partial V =(1-\rho_1-\rho_2)$ and $w=e^V v$. Then, $v$ is the weak solution of
$$ \partial_x( e^V \partial_x v) = 0 $$ 
in $(0,1)$ with boundary conditions 
\begin{align*}
	&\epsilon  \partial_x v =  \alpha v & \text{at } x=0, \\
	&\epsilon \partial_x v =  -\beta v & \text{at } x=1. 
\end{align*}
Using the weak formulation of this boundary value problem with test function $v$ implies
$$ \int_0^1 e^V v^2~dx + \alpha e^{V(0)} v(0)^2 + \beta e^{V(0)} v(1)^2 = 0, $$
which yields $v \equiv 0$ and thus uniqueness of the solution. 
\end{proof}

We start our analysis of the flow properties with a simple calculation relating the difference of $\rho$ to the constant state $\frac{1}2$ to the boundary values:
\begin{lemma} \label{1dlemmax}
Let $\rho \in H^1([0,1])$ be the unique weak solution of \eqref {uniflow1deq}-\eqref{uniflow1dbc2}. Then the estimate 
\begin{equation}
\int_0^1 \left(\rho - \frac{1}2 \right)^2 ~dx + \beta \rho(1) - \frac{1}4  \leq \epsilon \vert 1 - \alpha - \beta \vert
\end{equation}
holds.
\end{lemma}
\begin{proof}
Using the test function $\varphi(x)=x$ in the weak form of \eqref {uniflow1deq} and adding and subtracting $\frac{1}4$ we find
$$ \int_0^1 (\epsilon \partial_x \rho  + (\rho - \frac{1}2)^2)~dx + \beta \rho(1) - \frac{1}4 = 0 .$$
Further integrating the first term and using the a-priori bounds from the maximum principle for the boundary values concludes the proof.
\end{proof}

Lemma \ref{1dlemmax} will yield the desired asymptotic estimate if we can guarantee that $\beta \rho(1) \geq \frac{1}4$, such that the second term on the left-hand side is nonnegative. Note also that in spatial dimension one the flux is constant, thus we find
$\beta \rho(1) = \alpha (1- \rho(0))$, i.e. the above result could equally be formulated in terms of $\alpha$ respectively the inflow boundary value. To prove the latter under appropriate conditions is the objective of the next result:

\begin{thm}[Maximal Current Phase] \label{1dflowthm1}
Let $\rho \in H^1([0,1])$ be the unique weak solution of \eqref {uniflow1deq}-\eqref{uniflow1dbc2} and let
$$ \min\{\alpha,\beta\} \geq \frac{1}2.$$ 
Then the estimate 
\begin{equation}
\int_0^1 \left(\rho - \frac{1}2 \right)^2 ~dx  \leq \epsilon \vert 1 - \alpha - \beta \vert
\end{equation}
holds and furthermore, we have $j \ge 1/4$.
\end{thm}
\begin{proof}
Using Lemma \ref{1dlemmax} it suffices to show $\beta \rho(1) \geq 1/4$, which we carry out by contradiction.
Assume $ \rho(1) =\frac{1}{4\beta} - \delta$ with $\delta > 0$. Since $\beta \geq \frac{1}2$ and $\alpha \geq \frac{1}2$ we conclude in particular
$$ \rho(1) =\frac{1}{4\beta} - \delta \leq \frac{1}2 \qquad \text{ and} \qquad \rho(0) =1 - \frac{\beta}\alpha \rho(1) \geq 1- \frac{1}{4\alpha} + \frac{\beta}\alpha \delta \geq \frac{1}2+ \frac{\beta}\alpha \delta. $$ 
Let $H$ be a smooth monotone function such that
$$ H(0)=0, \qquad H(1)=1, \qquad \text{supp}(H') \subset (\frac{1}2 - \gamma, \frac{1}2+\gamma) $$
with $\gamma <\min\{\delta,\frac{\beta}{\alpha}\delta\}$. 
Now we choose the test function
$ \varphi = H(\rho)$
in the weak form of \eqref {uniflow1deq} again with $\frac{1}4$ added and subtracted.
Then we find
$$ \int_0^1 \left( \epsilon H'(\rho) |\partial_x \rho|^2 - H'(\rho)\rho (1-\rho) \partial_x \rho \right)~dx + \beta \rho(1) (H(\rho(1)) - H(\rho(0)) = 0. $$
Using the nonnegativity of the first term and rewriting the second term yields
$$  (\beta \rho(1) - \frac{1}4) (H(\rho(1)) - H(\rho(0)) \leq - \int_0^1 H'(\rho) (\rho - \frac{1}2)^2 \partial_x \rho ~dx 
= F(\rho(0)) - F(\rho(1)), $$ 
where $F$ satisfies $F'(p) = H'(p) (\rho-\frac{1}2)^2$ and $F(0)=0$. With the properties of $H$ it is straightforward to see that 
$$ H(\rho(1)) - H(\rho(0)) = 1, \qquad F(\rho(1)) = 0, \qquad F(\rho(0)) \leq 3 \gamma^2 . $$
Hence, we conclude 
$$ - \delta = \beta \rho(1) - \frac{1}4 \geq - 3 \gamma^2 , $$
which is a contradiction for $\gamma$ sufficiently small. Since the flux is constant, the fact that $j \ge 1/4$ follows immediately from \eqref{eq:bc2}.
\end{proof}

We remark that the strategy of proof of Theorem \ref{1dflowthm1} is reminiscent of entropy solution concepts for conservation laws and parabolic equations (cf. \cite{karlsen}), where roughly speaking the Heaviside function applied to $\rho - c$ for arbitrary constant $c$ multiplied with a nonnegative smooth function is used as a test function to define entropy inequalities. The function $H$ in the above proof will indeed approximate the Heaviside function of $\rho - \frac{1}2$ as $\gamma$ tends to zero.

In the inflow- and outflow-limited case the analysis is easier, an estimate like in Lemma \ref{1dlemmax} suffices:
\begin{thm}[Inflow- and Outflow Limited Phases] \label{1dflowthm2}
Let $\rho \in H^1([0,1])$ be the unique weak solution of \eqref {uniflow1deq}-\eqref{uniflow1dbc2} and let
$$ \max\{\alpha,\beta\} < \frac{1}2.$$ 
Then for $\alpha <  \beta$ the estimate 
\begin{equation}
\int_0^1 |\rho - \alpha| ~dx \leq \epsilon \frac{1- \alpha - \beta}{\beta - \alpha} \end{equation}
holds, while for $\alpha > \beta$
\begin{equation}
\int_0^1 |\rho - 1 + \beta| ~dx \leq \epsilon \frac{1- \alpha - \beta}{\alpha-\beta} \end{equation}
\end{thm}
\begin{proof}
Note that the maximum principle implies $\alpha \leq \rho \leq 1-\beta$ in any of the two cases.
We only detail the case $\frac{1}2 > \beta \geq \alpha$, as the other one is analogous.  Using the test function $\varphi(x) = 1-x$ in the weak formulation and some rewriting we have
$$ 0 = \int_0^1 (\epsilon \partial_x \rho + \rho^2 - \rho )~dx + \alpha (1- \rho (0) ). $$
With some rearranging and the bounds on $\rho$ we have  
$$ \beta \int_0^1 (\rho - \alpha) ~dx \leq \int_0^1 (1- \rho)(\rho - \alpha)~dx \leq \epsilon (1-\beta-\alpha) + \alpha \int_0^1\rho~dx - \alpha \rho(1). $$
Since $\rho \geq \alpha$ we have
$$ (\beta-\alpha) \int_0^1 |\rho - \alpha| ~dx \leq \epsilon (1-\beta-\alpha). $$
\end{proof}
Summing up, we have shown exactly the same behaviour for our the continuum model as \cite{Wood2009} for the discrete TASEP.
%

\subsection{Explicit Solutions in One Spatial Dimension}\label{sec:explicit}
In this section, we briefly discuss explicit solutions to the one-dimensional equations \eqref {uniflow1deq}--\eqref{uniflow1dbc2}. This derivation is mostly based on basic calculus and the details are presented in the appendix. However, this approach allows us to clarify the role of the parameter $\eps$ with respect to the phase diagram. In particular, we can show that for $\eps > 0$, maximal current can occur for values of $\alpha,\,\beta$ that are strictly smaller than $1/2$. 
Since in one space dimension, the flux $j$ is constant, we can set $j=1/4$ and integrate \eqref{uniflow1deq} and obtain the first order ordinary differential equation 
\begin{align*}
 -\partial_x\rho + \rho(1-\rho) = \frac{1}{4},
\end{align*}
which is also known as ``logistic equation with harvesting'' in the context of population dynamics, cf. \cite{Brauer1975,Cooke1986}. Solving this equation subject to the boundary conditions elementary calculation shows that on of the following conditions on $\alpha$ and $\beta$ have to hold in order to obtain a continuous solution:
\begin{align}\label{eq:contex1}
 \frac{1}{2}\frac{1+2\eps}{4\eps+1} < \alpha &< \frac{1}{2}\quad\text{ and }\quad \beta = \frac{1}{2}\frac{4\alpha\eps+2\alpha-1}{8\alpha\eps+2\alpha-2\eps-1},\\\label{eq:contex2}
 \frac{1}{2}\frac{1+2\eps}{4\eps+1} < \beta &< \frac{1}{2}\quad\text{ and }\quad \alpha = \frac{1}{2}\frac{4\beta\eps+2\alpha-1}{8\beta\eps+2\beta-2\eps-1},
\end{align}
Interestingly, maximal flux is achieved for values of $\alpha$, $\beta < 1/2$ which is in contrast to the discrete model, \cite{Wood2009}. To illustrate this, we depicted the changes in the phase diagram for different values of $\eps$ in figure \ref{fig:phase_eps}. In section \ref{sec:num1d} we present numerical results based on a discretization of \eqref{uniflow1deq}--\eqref{uniflow1dbc2} that confirm this observation.

\begin{figure}
\centering
\includegraphics[width=0.8\textwidth]{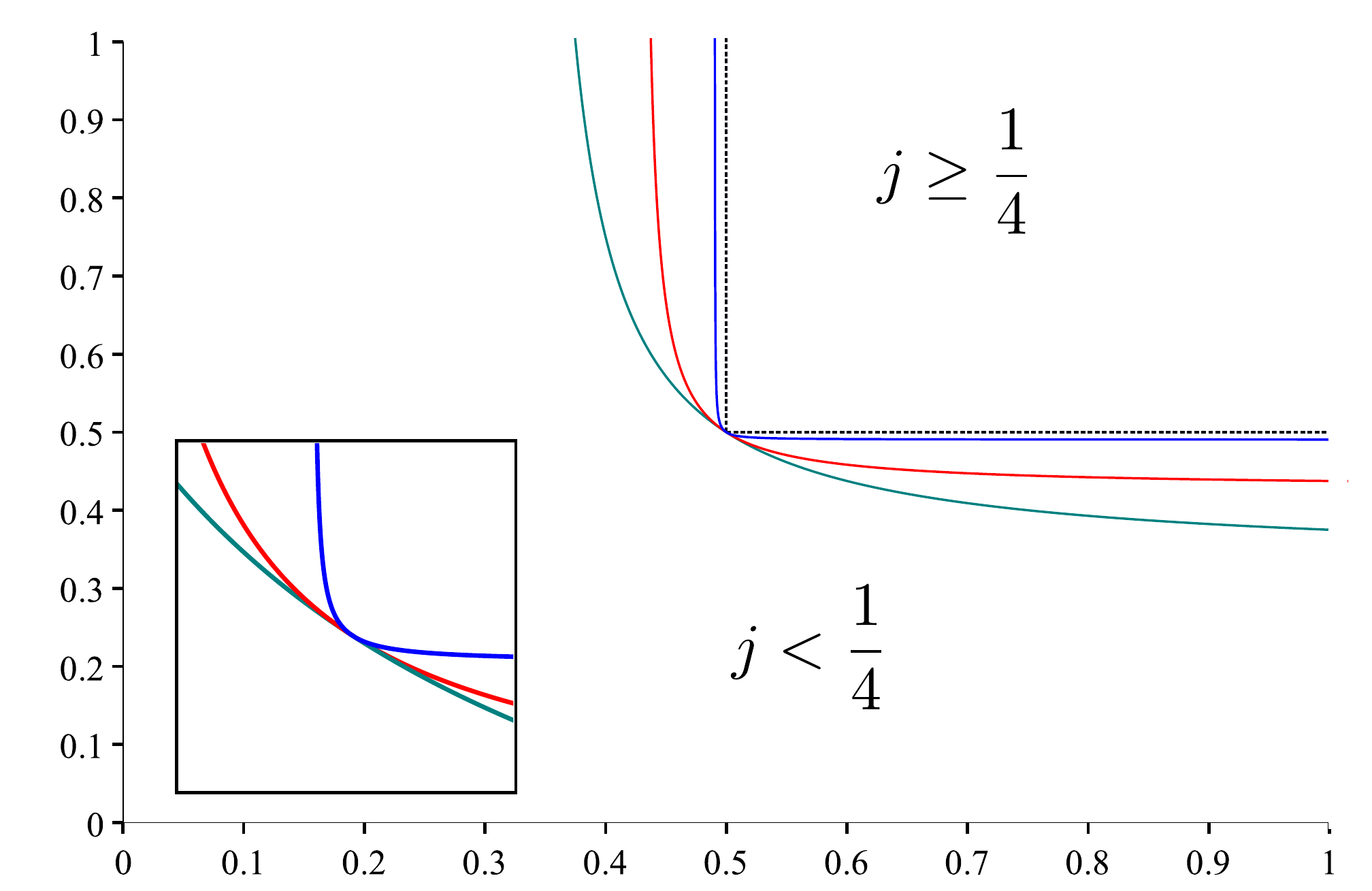}
\caption{In this phase diagram, the solid lines separate the regions of $j\ge 1/4$ (maximal flow) and $j< 1/4$ for the values $\eps = 0.01$ (blue), $\eps = 0.1$ (red) and $\eps = 0.5$ (green). The dashed lines correspond to the discrete case of \cite{Wood2009} where maximal flux is achieved for $\alpha,\,\beta > 1/2$, only. The inlet shows a magnification of the area around the point $(\alpha,\beta)=(1/2,1/2)$.}
\label{fig:phase_eps}
\end{figure}

\section{Numerical Solution}\label{sec:numerics}
In this section we will describe the numerical method that we used and present some examples in one and two space dimensions. Our implementation is based on the discontinuous finite element method which is well-suited for convection dominated problems, see \cite{DiPietro2012} and the references therein. We will not give any details regarding error estimates and the convergence of our algorithm which remains future work.

\subsection{Setting and Discontinuous Galerkin Scheme}\label{sec:scheme}
Let us recall some well-known notations and definitions, cf. \cite{DiPietro2012}. We start by dividing our domain into elements which are triangles in two space dimensions and intervals in 1D. For simplicity, we shall only discuss the two-dimensional case from now on. We cover the domain $\Omega \subset \RR^2$ by a finite collection of triangles which we denote by $\mathcal{T}_h$, where $h$ refers to the diameter of the largest triangle. Furthermore, we denote by $F$ the mesh faces which are characterised by one of the following two conditions:
\begin{enumerate}
 \item Either, there are distinct triangles $T_1$ and $T_2$ such that $F = \partial T_1 \cap \partial T_2$ - $F$ is a interface,
 \item or, there is $T\in \mathcal T_h$ such that $F = \partial T \cap \partial\Omega$ - $F$ is a boundary face.
\end{enumerate}
We denote by $\mathcal F_h^i$ the set of all interfaces, $\mathcal F_h^b$ the boundary faces and by $\mathcal F_h$ the union of these two sets. Furthermore, $ n_F$ is the normal vector of a facet, pointing outward. On $\mathcal{T}_h$ we introduce the broken polynomial space
\begin{align*}
 V_h = \{ v \in L^2(\Omega) \;:\; \forall\, T \in\mathcal{T}_h,\, \left. v \right|_T \in \mathcal{P}^1(T)\,\},
\end{align*}
where $\mathcal{P}^1(T)$ denotes polynomials of degree one on $T$. 
For a scalar function $v$, smooth enough for the expression $\left. v\right|_{F}$ for all $F \in \mathcal F$ to make sense, we define interface averages and jumps in the following way
\begin{align*}
 \avg{v}_F(x) &:= \frac{1}{2} ( \left. v\right|_{T_1}(x) + \left. v\right|_{T_2}(x) ), \text{ for a.e. }x\in F, &\text{(average)},\\
 \jump{v}_F(x) &:=\left. v\right|_{T_1}(x) - \left. v\right|_{T_2}(x),\text{ for a.e. }x\in F, &\text{(jump)}.
\end{align*}
With these definitions at hand, we can state our discontinuous Galerkin scheme. Starting from the weak formulation of a linearized version of \eqref{eq:basic} we consider
\begin{align}\label{eq:weak}
 \underbrace{\eps\int_\Omega\nabla \rho\nabla\phi\d x + \int_\Omega \rho(1-\tilde\rho)u\nabla\phi\d x}_{=:a(\rho,\phi;\tilde\rho)} + \underbrace{\alpha\int_{\Gamma}\rho\phi\d s + \beta\int_{\Sigma}\rho\phi\d s}_{=: a_F(\rho,\phi)} = \underbrace{\alpha\int_\Gamma \phi\d s}_{=:f(\phi)},\quad\phi_h \in H^1(\Omega),
\end{align}
with $\tilde\rho \in H^1(\Omega)\cap L^\infty(\Omega)$ given. In order to obtain a discrete solution $\rho_h \in  V_h$ we define the bilinear form
\begin{align*}
a(\rho_h,\phi_h ; \tilde \rho ) = a^{\rm swip}(\rho_h,\phi_h) + a^{\rm upw}(\rho_h,\phi_h ; \tilde \rho),
\end{align*}
with a symmetric weighted interior penalty method for the diffusion given by
\begin{align*}
 a^{\rm swip}(\rho_h,\phi_h) &= \int_\Omega \eps \nabla_h\rho_h\cdot\nabla_h\phi_h\d x  - \sum_{F\in\mathcal{F}_h} \eps\int_F \left(\avg{\nabla_h\rho_h}\cdot  n_F\jump{\phi_h} + \jump{\rho_h}\avg{\nabla_h\phi_h}\cdot  n_F\right)\d\sigma \\
 &+ \sum_{F\in\mathcal{F}_h}\eta\frac{\eps}{h_F}\int_F\jump{\rho_h}\jump{\phi_h}\d\sigma
\end{align*}
and a upwind scheme for the advection part
\begin{align*}
a^{\rm upw}(\rho_h,\phi_h) &= \int_\Omega -\rho_h((1-\rho_h)u\cdot\nabla_h \phi_h)\d x  + \sum_{F\in\mathcal{F}_h^i}\int_F ((1-\tilde \rho_h)u\cdot  n_F \avg{\rho_h} \avg{\phi_h}\d\sigma\\
&+ \sum_{F\in\mathcal{F}_h^i}\frac{1}{2}|(1-\rho_h)u\cdot  n_F|\int_F \jump{\rho_h}\jump{\phi_h}\d\sigma,
\end{align*}
again with $\tilde \rho_h \in V_h$ given. The local length scale $h_F$ is defined as $h_F = \frac{1}{2}(h_{T_1} + h_{T_2})$, where $T_1$ and $T_2$ are the two triangles adjacent to face $F$. In order to obtain a solution to the original nonlinear problem \eqref{eq:basic} we employ the following semi-implicit iteration scheme: For $u^n_h$ given find $u_h^{n+1}\in V_h$ s.t.
\begin{align}\label{eq:scheme}
 (u_h^{n+1},\phi_h) + \tau(a(u_h^{n+1},\phi_h ; u_h^n) + a_F(u_h^{n+1},\phi_h)) = (u_h^n,\phi_h) + f(\phi_h),\quad \forall \phi_h \in V_h,
\end{align}
with a relaxation parameter $\tau > 0$. Thus in each step one has to solve the following system of linear equations
\begin{align*}
 (M+\tau A) \underline{u}^{n+1} = (M\underline{u}^n+\tau\underline{f}),
\end{align*}
where $\underline{u}$ denotes the vector of coefficient of $u$ in the linear finite element basis, $A$ is the matrix corresponding to the bilinear form $(a + a_F)$, and $M$ denotes the mass matrix. The vector $\underline{f}$ stems from the term $f(\phi_h)$ on the r.h.s. of \eqref{eq:scheme} with $u^n$ being the solution of the previous step. In all experiments below we chose $u_0=1/2$ and $\tau=0.01$. Note that this scheme can be interpreted as a semi-implicit time discretization of the parabolic version of \eqref{eq:basic} with time step size $\tau$.

\subsection{Results in one spatial dimension}\label{sec:num1d}
In one space dimension, we used MATLAB to implement the scheme described above. We will present several examples in the following and consider the domain $\Omega = [0,1]$ discretized by $n=200$ elements. 

\subsubsection*{Different phases}
First we present some examples to illustrate the occurrence of the three different phases (namely \emph{influx limited}, \emph{outflux limited} and \emph{maximal current}) that are analysed in section \ref{sec:asymptotic} (and also in \cite{Wood2009}). We performed simulations for $\eps = 0.1,\, 0.01,\, 0.001$. For $\alpha$ and $\beta$ we chose the values $0.2,\,0.4,\,0.6$ and  $0.4,\,0.2,\,0.7$, respectively. The numerical results confirm the predicted occurrence of three phases and the results are shown in figure \ref{fig:phases}.

\begin{figure}
\centering
\subfigure{\includegraphics[width=0.45\textwidth]{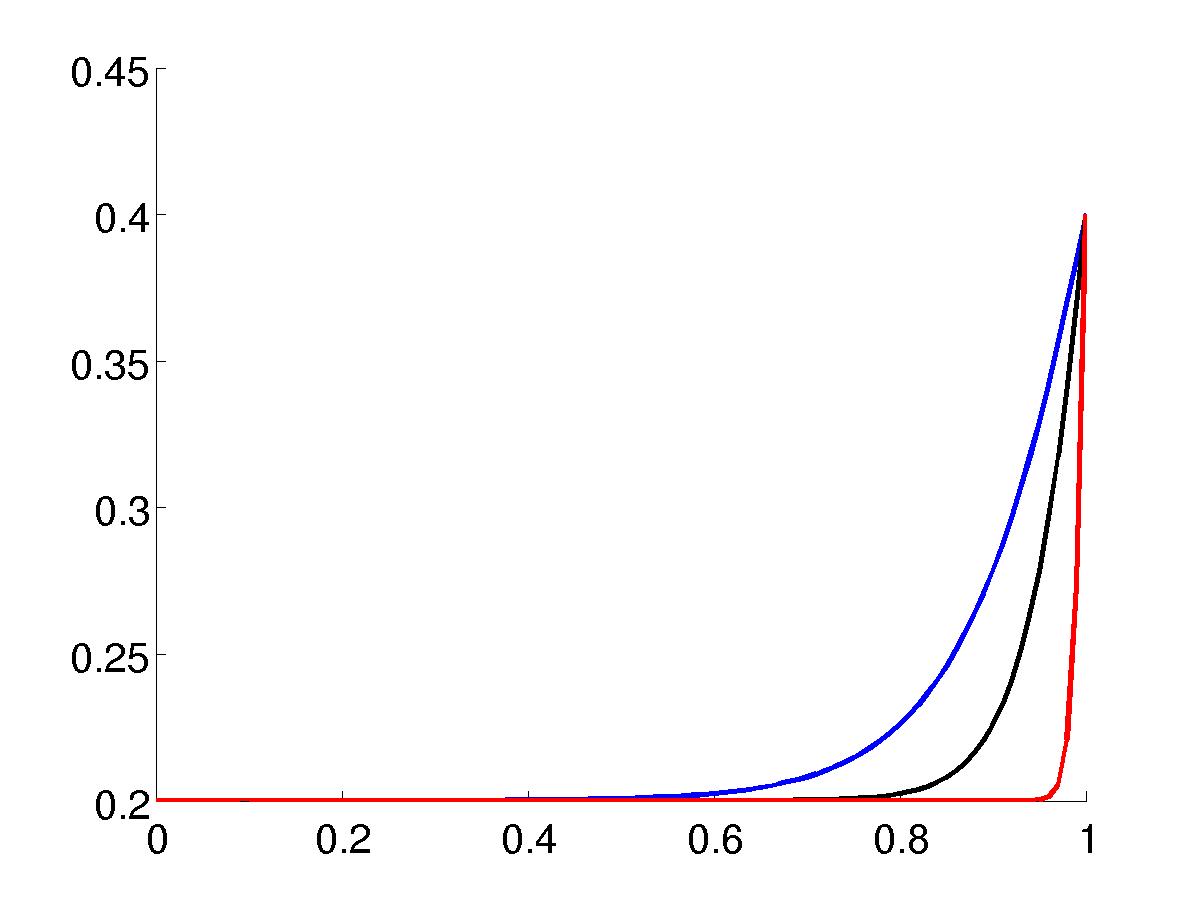}}
\subfigure{\includegraphics[width=0.45\textwidth]{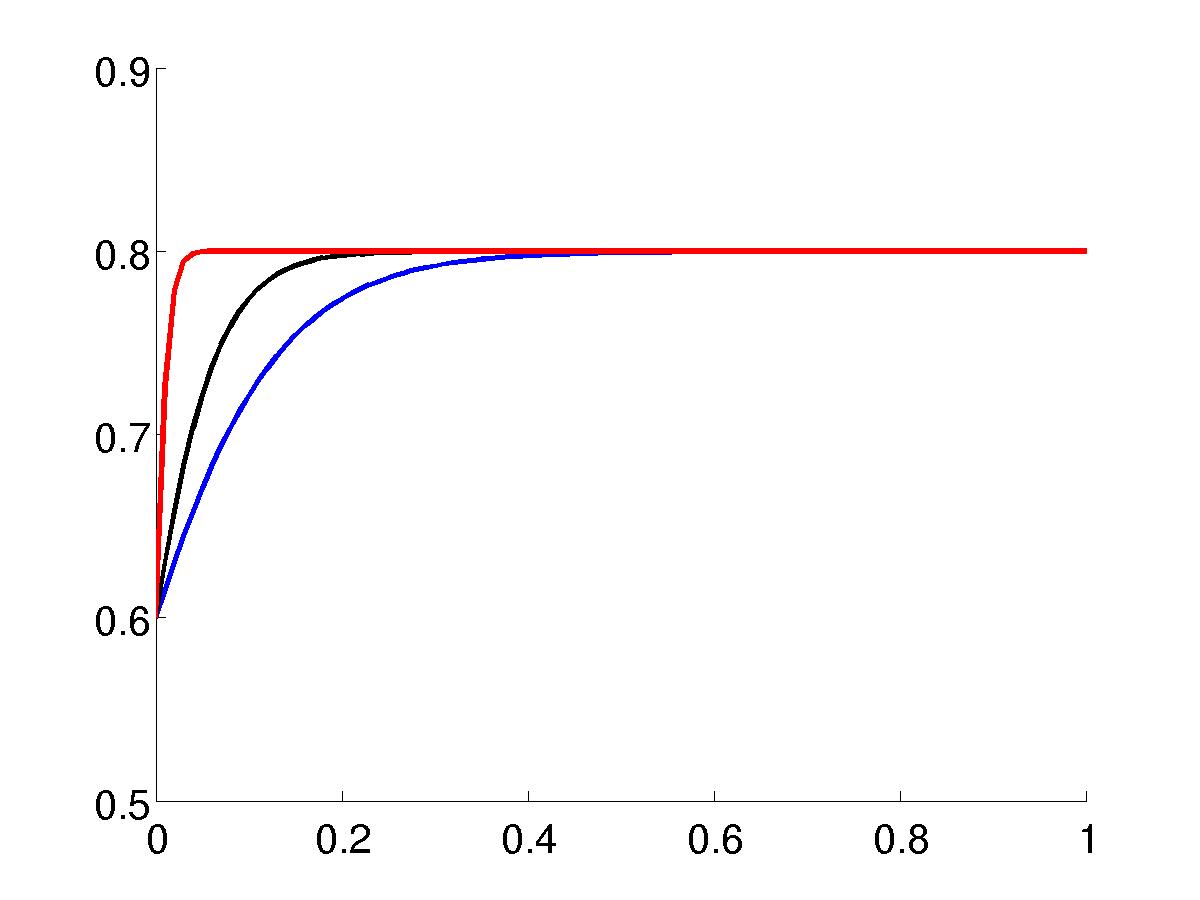}}\\
\subfigure{\includegraphics[width=0.45\textwidth]{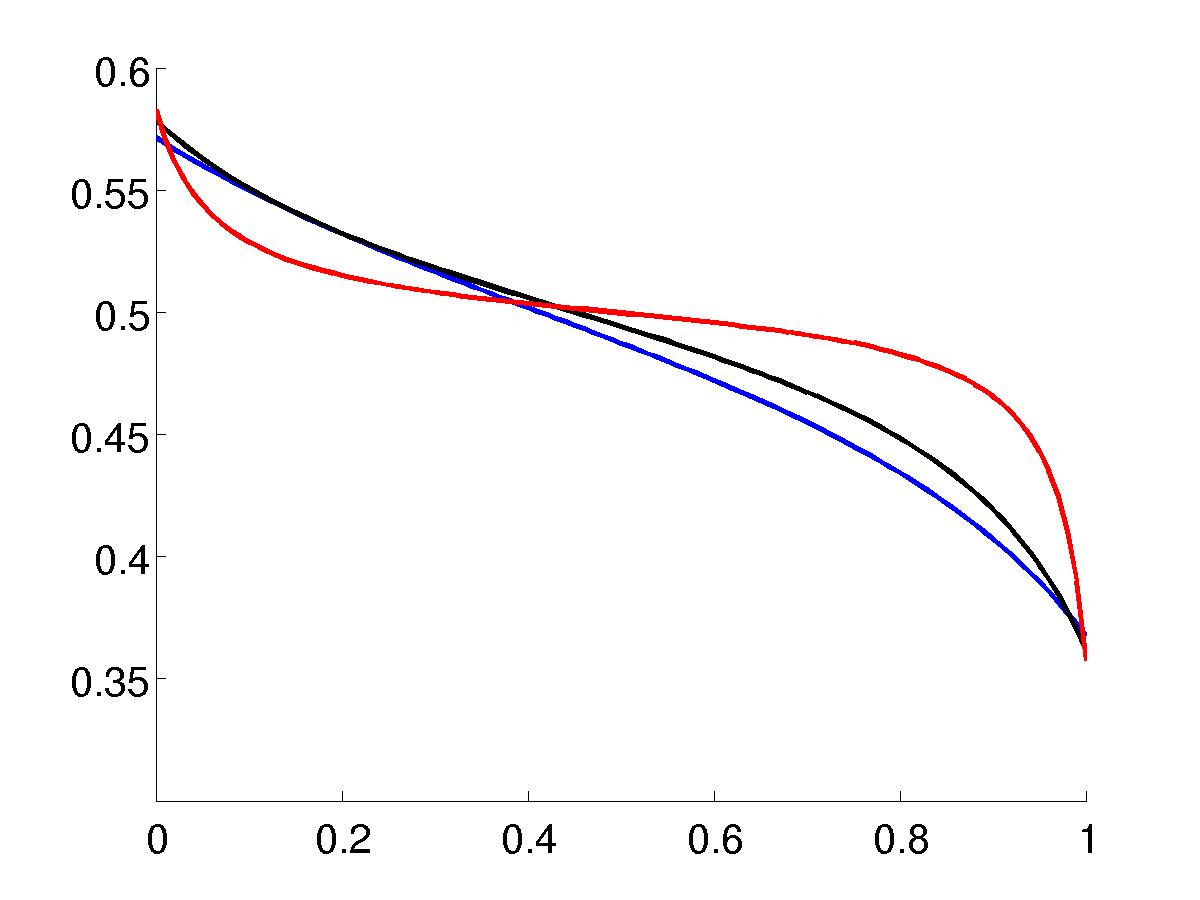}}
\subfigure{\includegraphics[width=0.45\textwidth]{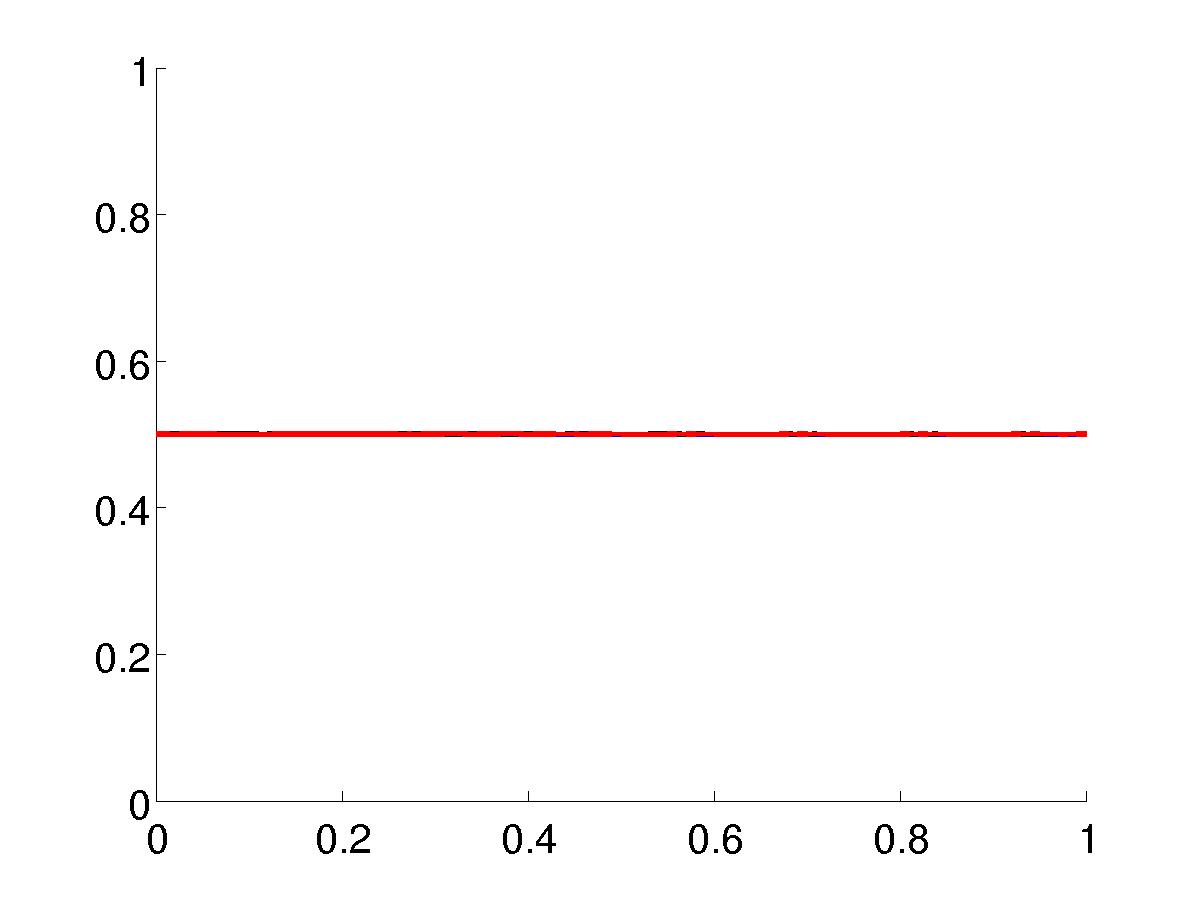}}\\
\caption{Some results of the 1D code for $\epsilon=0.1$ (blue),  $\epsilon=0.05$ (black), and $\epsilon=0.01$ (red).Top left: $\alpha=0.2$, $\beta=0.4$, Top right: $\alpha=0.4$, $\beta=0.2$, Bottom left: $\alpha=0.6$, $\beta=0.7$, Bottom right: $\alpha=0.5$, $\beta=0.5$}
\label{fig:phases}
\end{figure}
%

\subsubsection*{Maximal current for $\alpha<1/2$ or $\beta<1/2$}
Here we present numerical evidence for the results of section \ref{sec:explicit}, namely the occurrence of the maximal flow phase for $\alpha,\,\beta < 1/2$. We chose $\epsilon=0.01$ which yields $\frac{1}{2(2\epsilon+1)}=0.4902$. We chose $\alpha=0.4912$ and by \eqref{eq:contex1}, the corresponding $\beta$ is $0.603773585$. To illustrate the case $\beta < 1/2$ we simply interchange the roles of $\alpha$ and $\beta$. Both results are depicted in Figure \ref{fig:alpha_jmax} and confirm the results from \ref{sec:explicit}.
\begin{figure}
\centering
\subfigure{\includegraphics[width=0.45\textwidth]{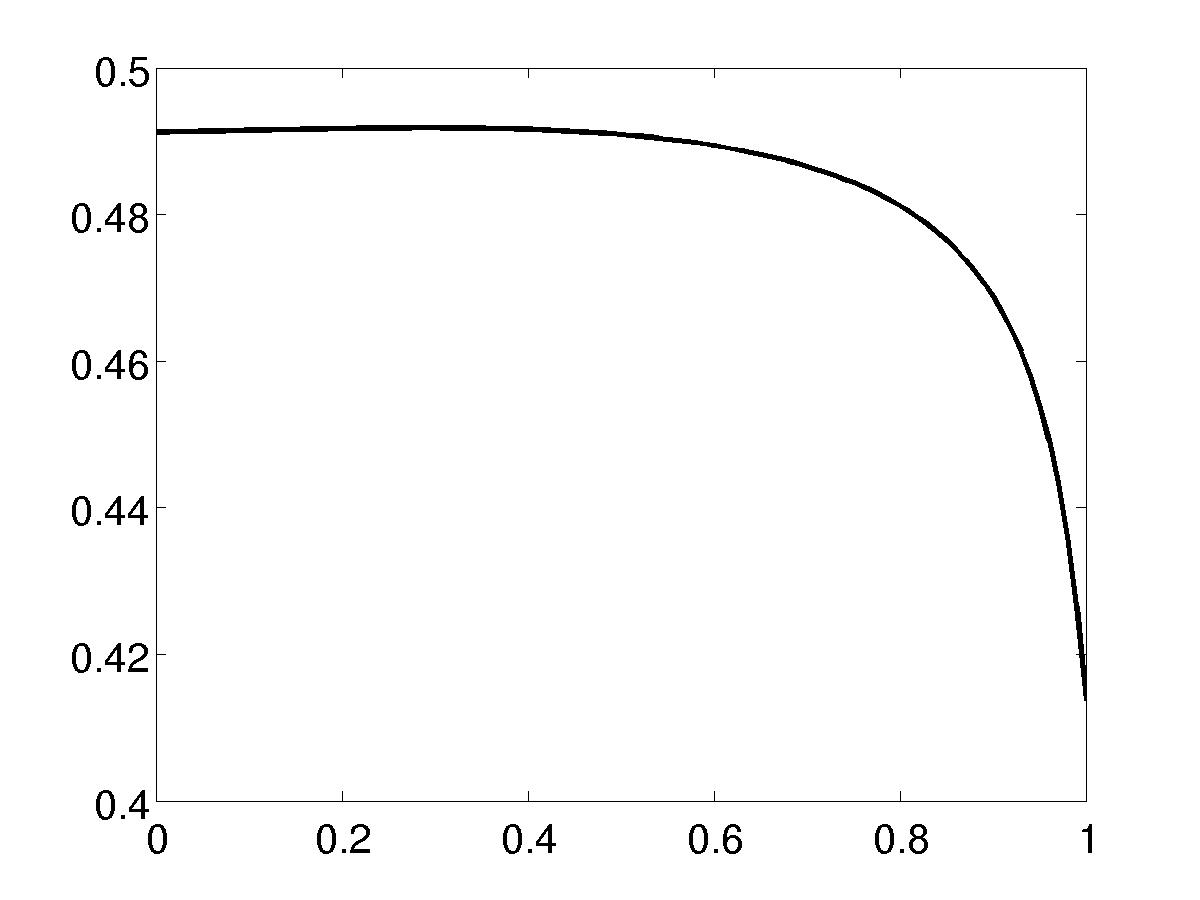}}
\subfigure{\includegraphics[width=0.45\textwidth]{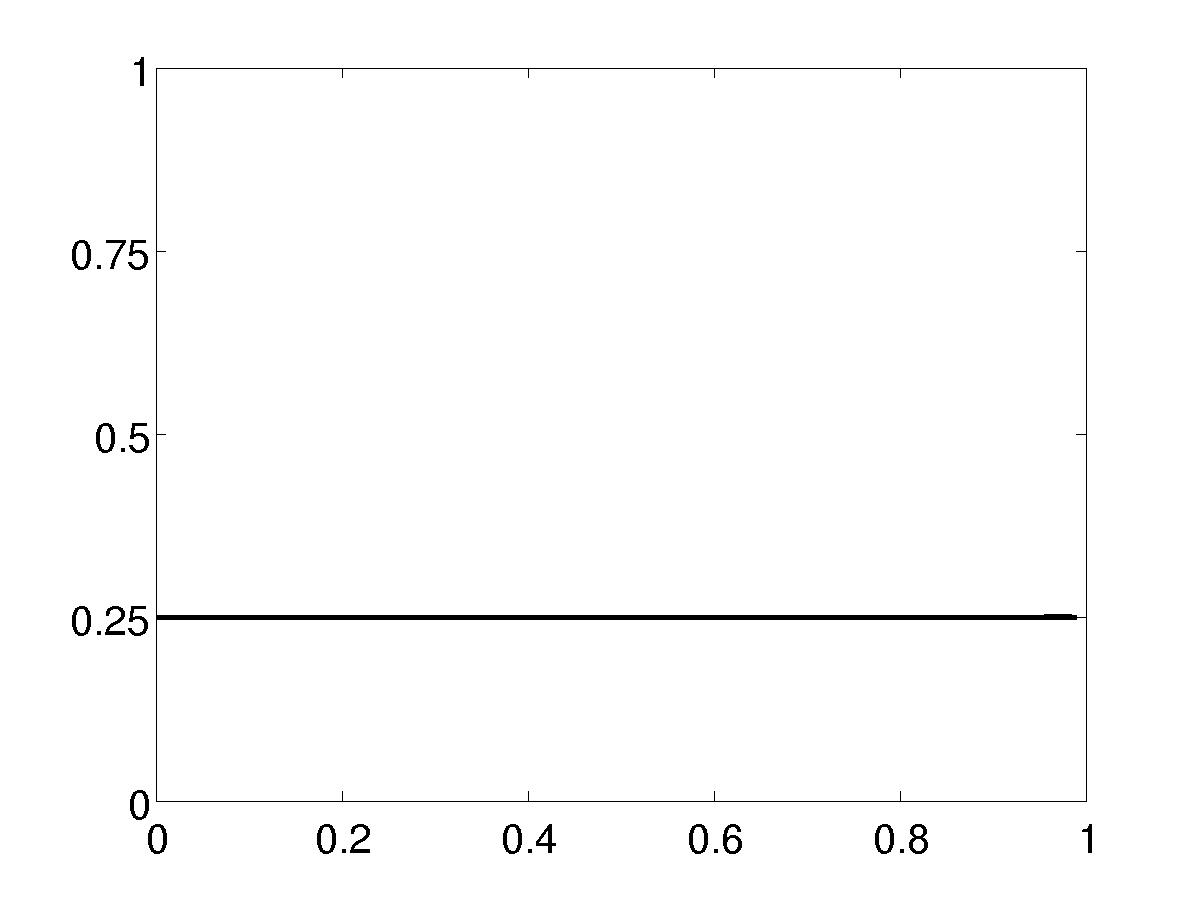}}\\
\subfigure{\includegraphics[width=0.45\textwidth]{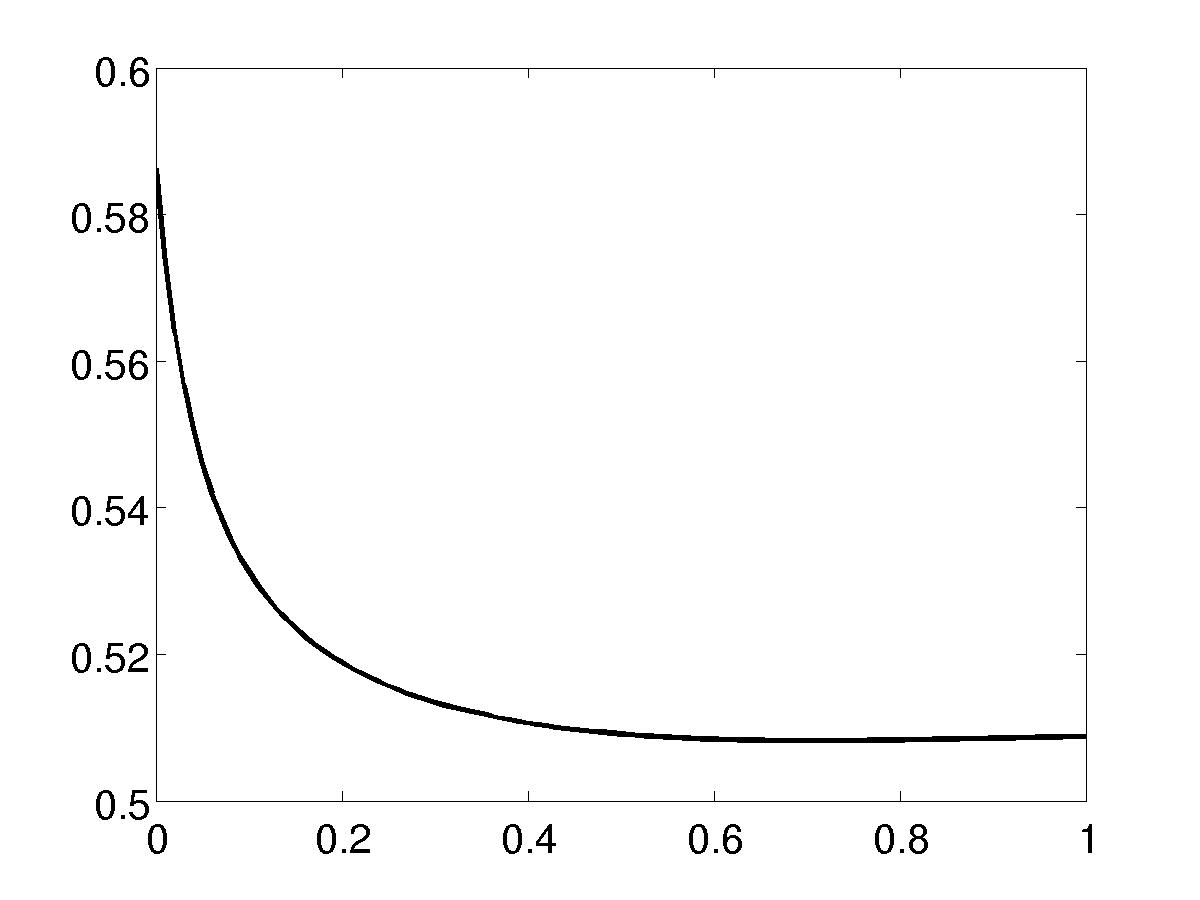}}
\subfigure{\includegraphics[width=0.45\textwidth]{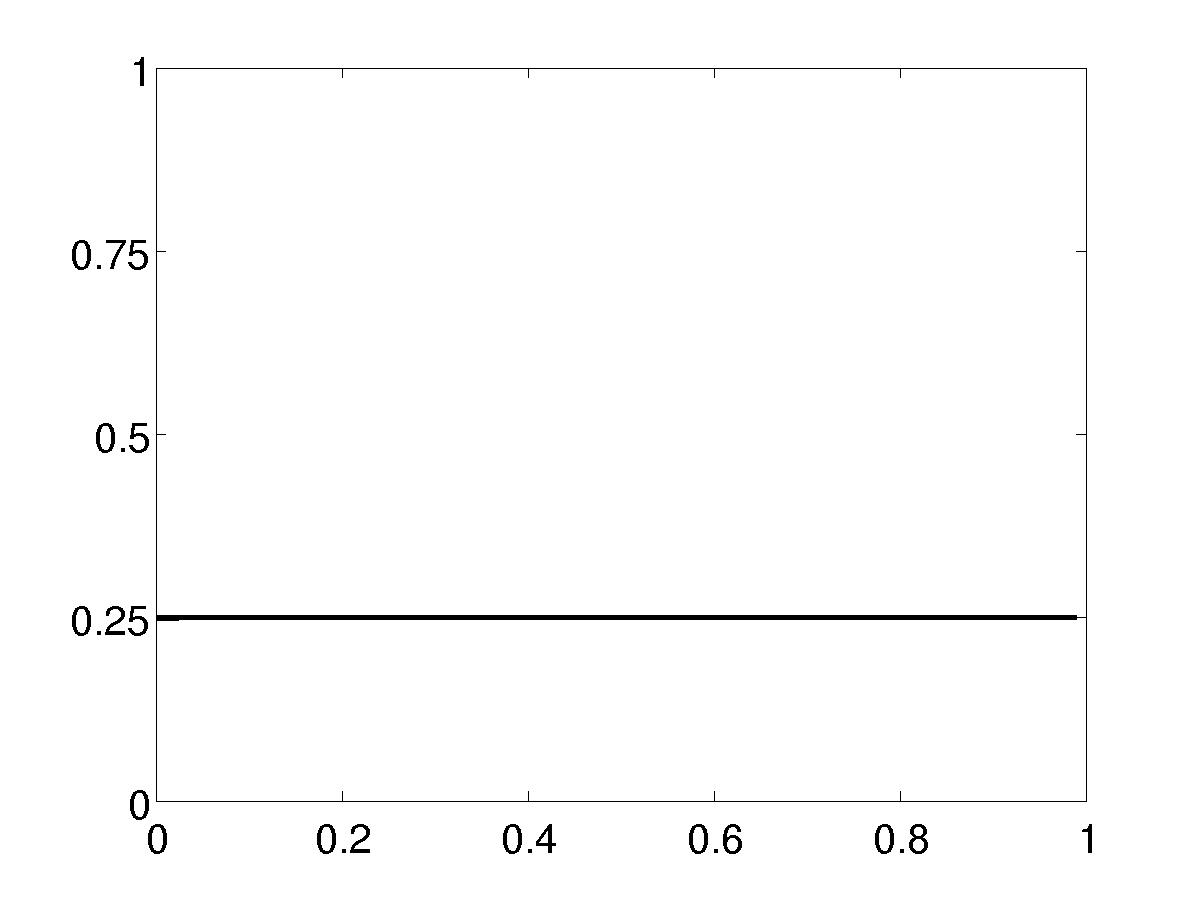}}\\
\caption{For $\eps = 0.01$, the resulting density (left) and flux (right) is depicted for the values $\alpha=0.4912$, $\beta=0.6043$ (top) and $\alpha=0.6043$, $\beta=0.4912$ (bottom). The maximal flux $j=1/4$ is observed in both cases.}
\label{fig:alpha_jmax}
\end{figure}
To further explore the behaviour of the flux, we used the discontinuous Galerkin scheme introduced above to numerically produce a phase diagram by sampling the values of $\alpha,\,\beta$ from $0$ to $1$ with a stepsize of $0.01$. For $\eps = 0.1$ we compared the contour line $j=1/4$ with \eqref{eq:contex1} or \eqref{eq:contex2}, respectively, see figure \ref{fig:phasenum}.
\begin{figure}
\centering
\subfigure{\includegraphics[width=0.49\textwidth]{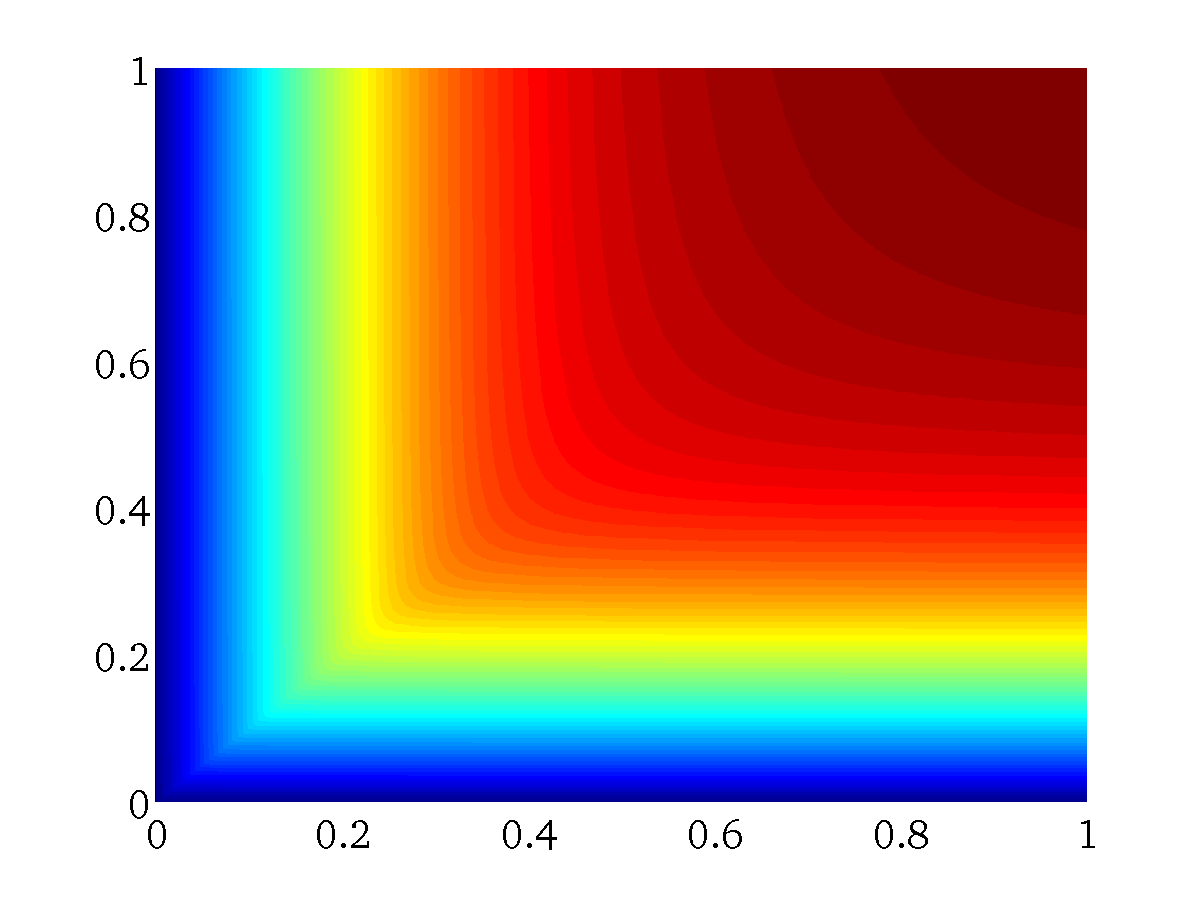}}
\subfigure{\includegraphics[width=0.49\textwidth]{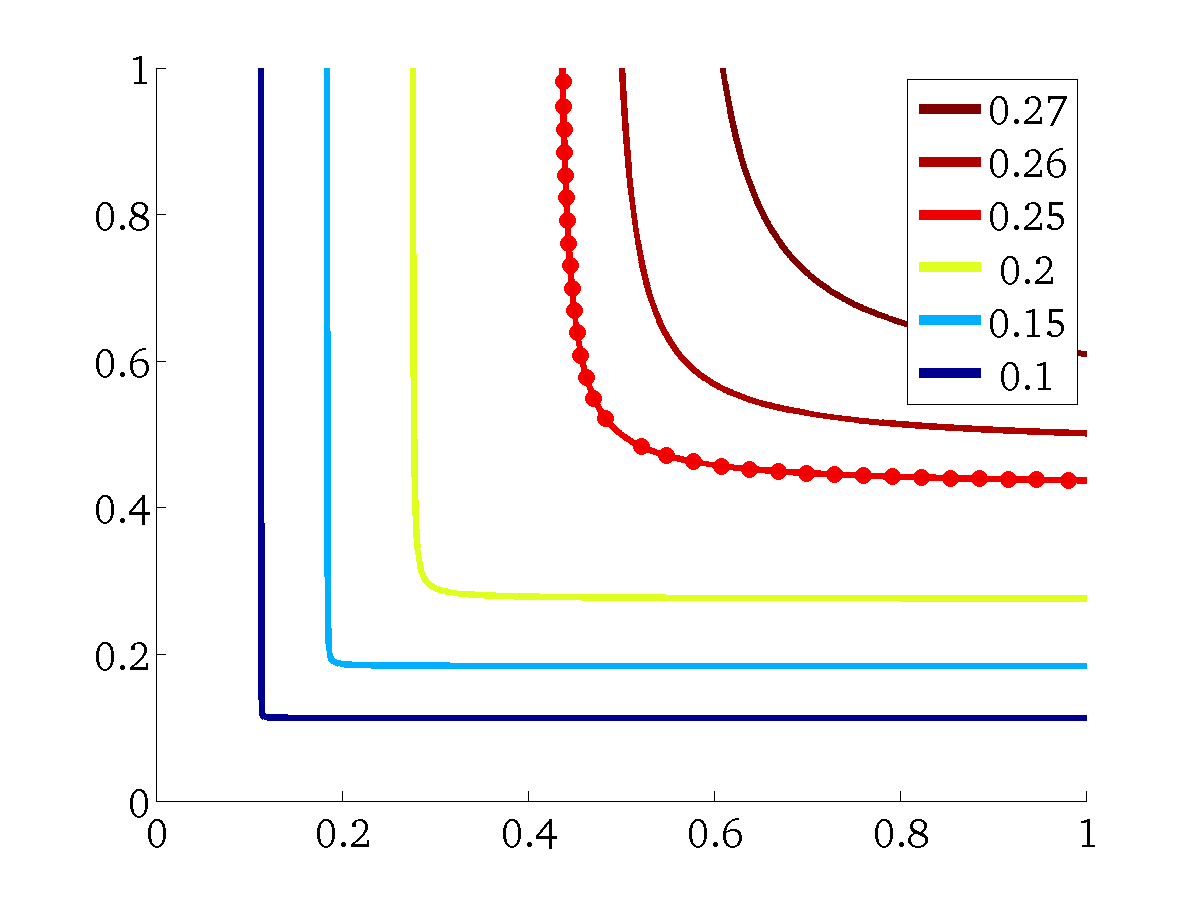}}\\
\caption{For $\eps = 0.1$, the left picture shows a phase diagram generated using the discontinuous Galerkin method described above for values of $\alpha,\,\beta = 0,0.001,0.002,\ldots,1$. On the right side, contour lines for several values of the flux $j$ are depicted. The line for $j=1/4$ (blue line) is compared to the analytical results of section \ref{sec:explicit} (red circles). Compare also with figure \ref{fig:phase_eps}.}
\label{fig:phasenum}
\end{figure}

\subsection{Results in two spatial dimensions}
In two spatial dimensions, we used the software package FeniCS, \cite{Logg2010,Logg2012} to implement the scheme described in section \eqref{sec:scheme}. We present several examples using the domain sketched in figure \ref{fig:sketch2d}, i.e. a corridor of length $2$ and height $1$ with two entrances and two exists on each side. The upper entrace and exit are located at $0.65 < y < 0.85$, the lower ones at $0.15 < y < 0.35$. For each entrance, we have a different inflow rates $\alpha_i$, $i=1,2$, while we have $\beta_i$, $i=1,2$ for the exits.
\subsubsection*{Maximum principle}
In all examples in this section, we use a velocity field given as the gradient of some potential. From theorem \ref{thm:M1_potential} and corollary \ref{cor:M1_max} we know that for general potentials $V$ we only have $0\le \rho \le 1$ while for $V$ satisfying the assumptions of corollary \ref{cor:M1_max} we have that $\min\{\alpha,1-\beta\} \le \rho(x) \le \max \{ \alpha, 1-\beta \}$. To illustrate this, let $V_m$ be given as the solution to the equation
\begin{align*}
 -\Delta V_m &= 0\;\text{in}\;\Omega, \\
 \partial_n V_m &= -1\;\text{ on }\Gamma,\\
 \partial_n V_m &= 1\;\text{ on }\Sigma,\\
 \partial_n V_m&=0\;\text{ on }\partial \Omega \setminus (\Gamma \cup \Sigma),
\end{align*}
with the normalization condition $\int_{\partial\Omega} u \;d\sigma(x) = C$. On the discrete level, we use a mixed method to discretize this equation in order to ensure the condition $\nabla\cdot \nabla V_m$ is fulfilled exactly on the discrete level. The normalization constrained is achieved by setting an arbitrary boundary node to zero. The resulting velocity field $u_m = \nabla V_m$ is depicted in Figure \ref{fig:nablau}. Alternatively we chose $V_l(x)=x$ which yields $u_l = \nabla V_l = (1,0)^t$. In our first example we then chose $\alpha_1 = 0.2$, $\alpha_2 = 0.4$, $\beta_1 = 0.4$ and $\beta_2 = 0.2$ and apply our scheme with $V=V_l$ and $V=V_m$, respectively. The results shown in figure \ref{fig:2dlane} produce the expected behaviour, namely the maximum principle \eqref{rhobounds} only occurs for $V=V_m$. Furthermore, the results show that the asymmetric in- and outflow rates indeed some of the ``particles'' entering at $\alpha_2$ move over two the exit at $\beta_1$. This indicates that the model may be able to also predict lane formation in the case with more than one active species (i.e. \eqref{basiceqn1} with $M>1$), see also \cite{schlakepietschmann}. 

\subsubsection*{High densities and obstables}
In a second example, we explored the situation of maximal flow by using the values $\alpha_1 = 0.6$, $\alpha_2 = 0.9$, $\beta_1 = 0.9$ and $\beta_2 = 0.6$ and $V=V_l$. Note that in both cases we observe on the parts between the in- and outflow boundaries that the maximum principle of theorem \ref{thm:incompressible} does not hold since $u_l \cdot n \neq 0$ on the no-flux boundary. The results are shown in figure \ref{fig:2dmax}. 
Since this example shows that high densities can occur between the two exits, we modify the domain by adding an round obstacle in front of the doors as shown in figure \ref{fig:2dhole}. This is motivated by results from models for human crowd motions where in some situations, an obstacle in front of the exits can improve the situation. Indeed, our results show that the densities between the two exits decreases, however at the price of a large density in front of the obstacle itself. Furthermore, the transition from high to low densities observed in figure \ref{fig:2dlane} is shifted towards the entrances. 
\begin{figure}
\centering
\includegraphics[width=0.7\textwidth]{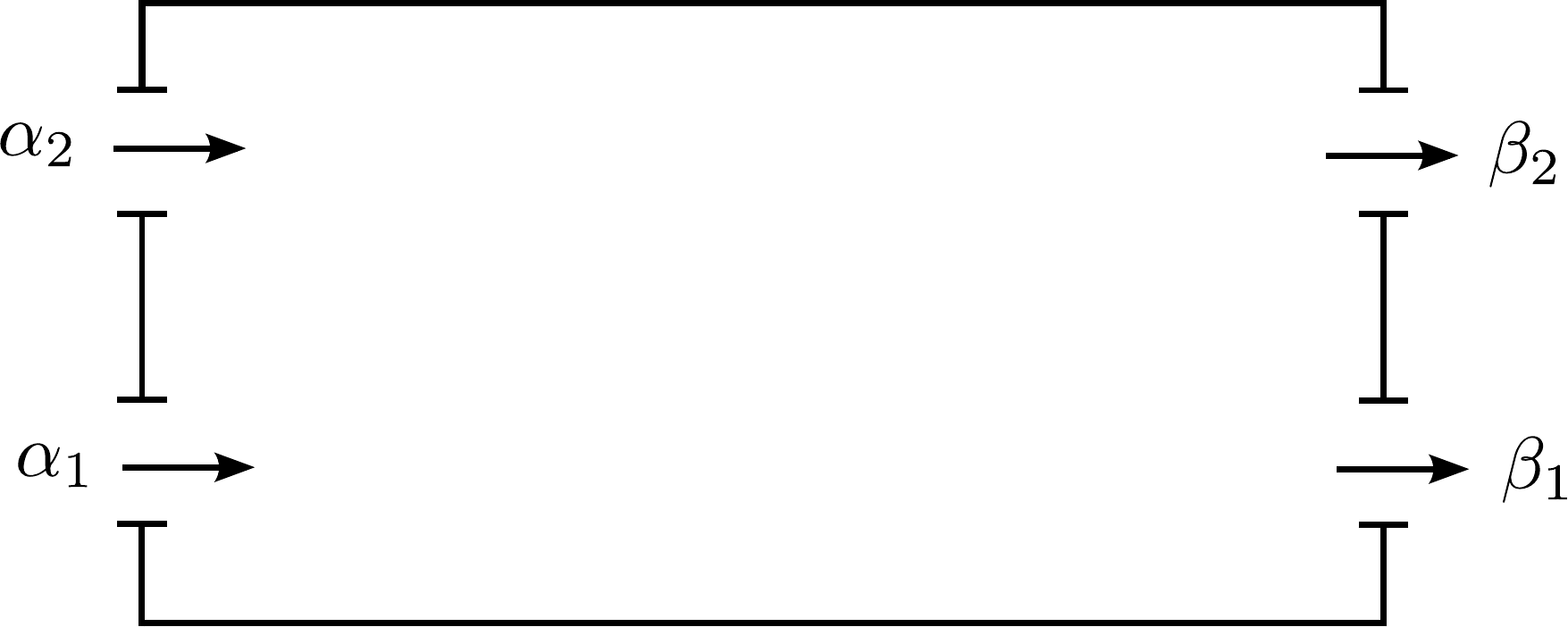}
\caption{Sketch of the geometry for the two-dimensional simulations. We consider a corridor of length $2$ and height $1$ and with two entrances and to exits, solated on the left and right boundary, respectively.}
\label{fig:sketch2d}
\end{figure}

\begin{figure}
\centering
\includegraphics[width=0.90\textwidth]{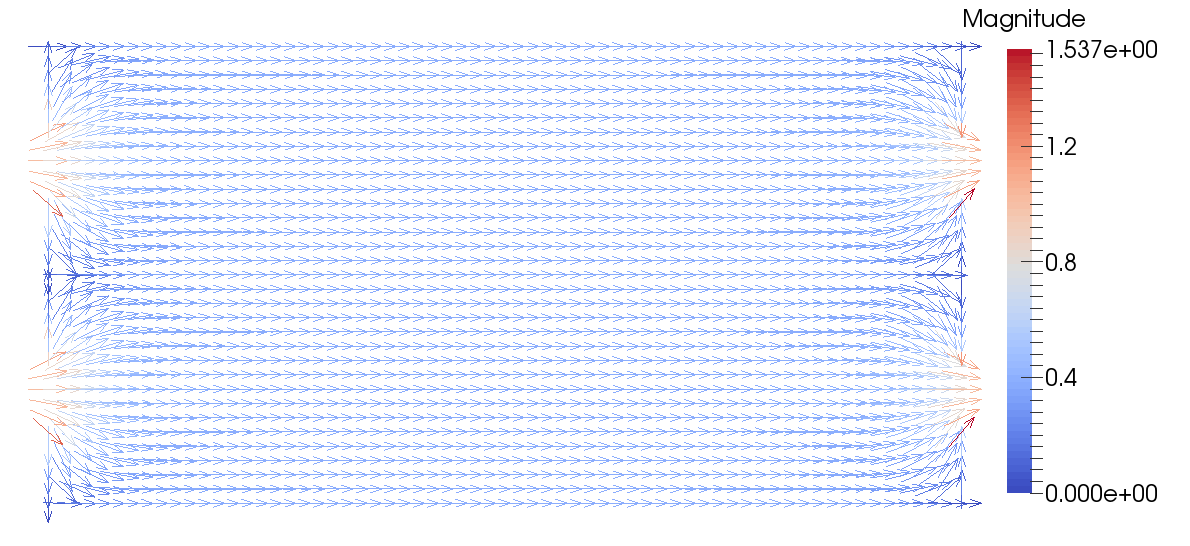}
\caption{The vector field $\nabla V_m$. The non homogeneous Neumann boundary conditions yield a velocity field that transports density away from the entrances and towards the exits thus facilitating the tranport.}
\label{fig:nablau}
\end{figure}

\begin{figure} 
\centering
\includegraphics[width=0.9\textwidth]{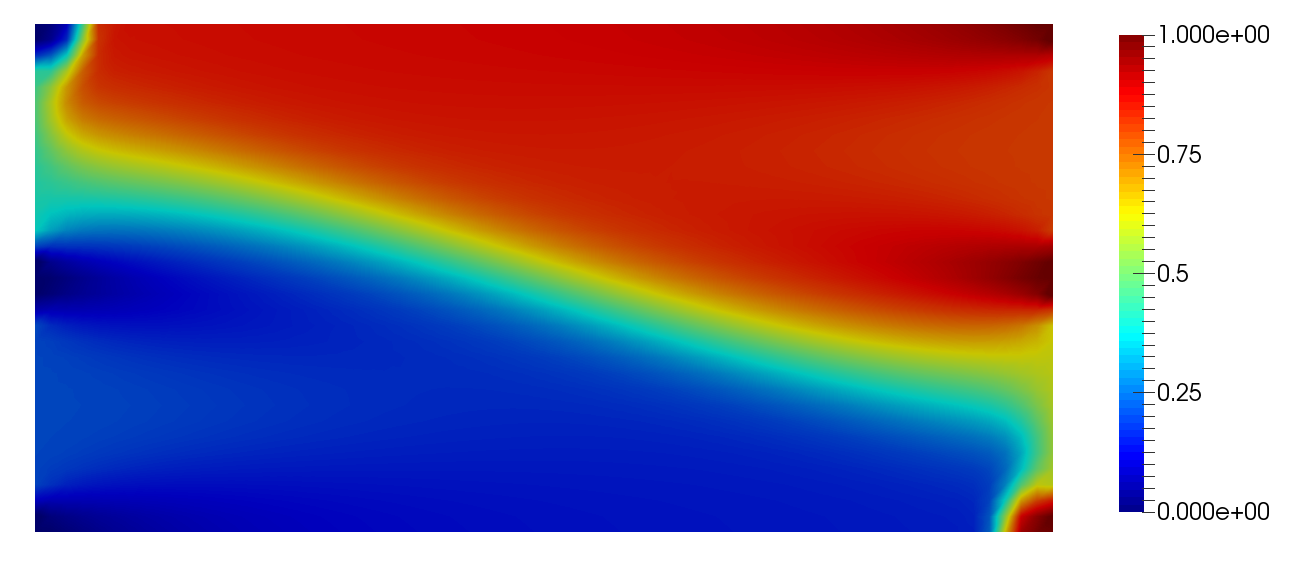}
\includegraphics[width=0.9\textwidth]{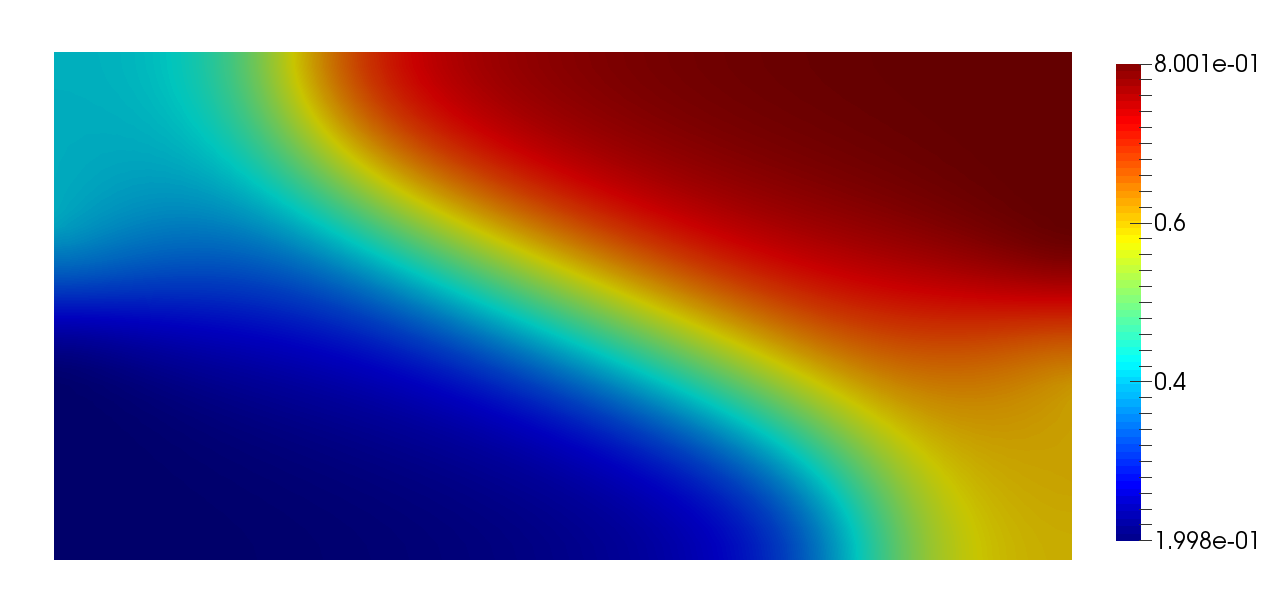}
 \caption{Simulation with rates $\alpha_1 = 0.2$, $\alpha_2 = 0.4$, $\beta_1 = 0.4$ and $\beta_2 = 0.2$. Above the results for $V=V_l$ are shown and even though the velocity field points in $x$-direction only, density is transported towards the larger exit. For $V=V_m$ (below) one clearly sees that the maximum principle is satisfied while still some density is transported to the lower exit. }
\label{fig:2dlane}
\end{figure}

\begin{figure} 
\centering
\includegraphics[width=0.9\textwidth]{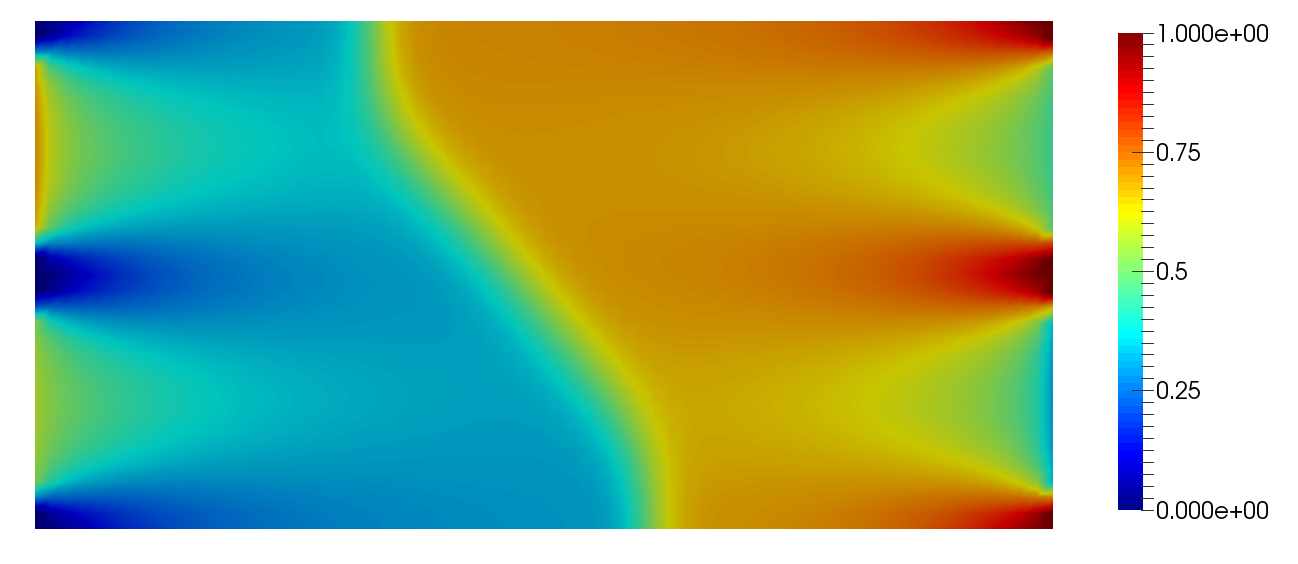}
\caption{Simulation with $V=V_l$ and rates $\alpha_1 = 0.6$, $\alpha_2 = 0.9$, $\beta_1 = 0.9$ and $\beta_2 = 0.6$, i.e. in the regime of maximal flow.}
\label{fig:2dmax}
\end{figure}

\begin{figure} 
\centering
\includegraphics[width=0.9\textwidth]{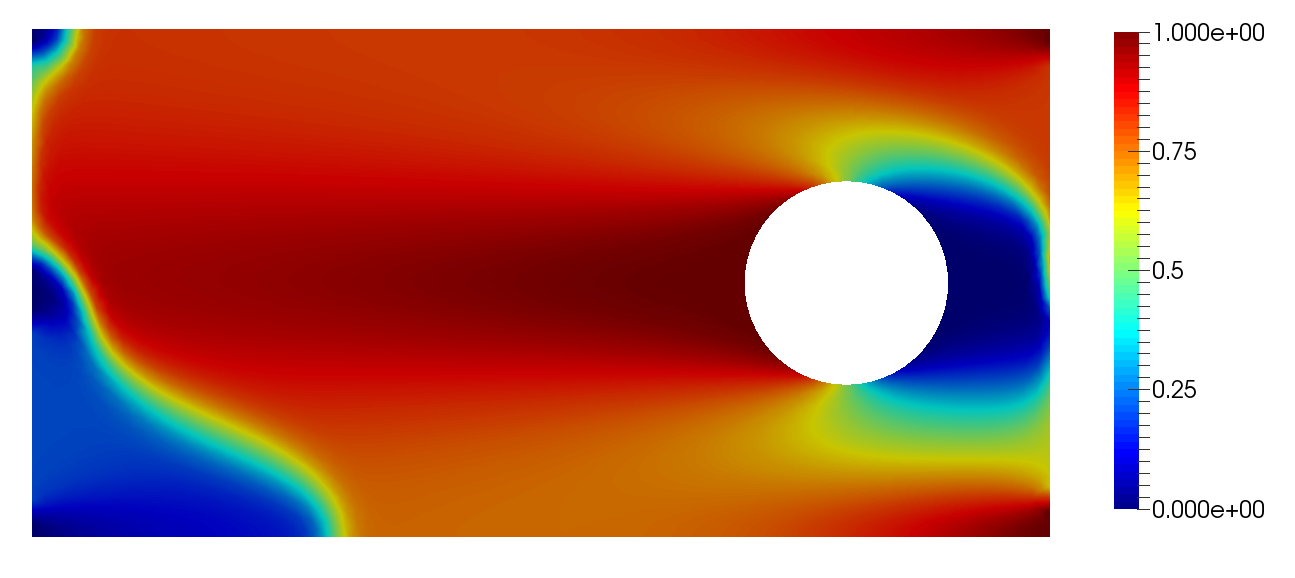}
 \caption{Introducing an obstacle dramatically decreases the density in front of the two exists, however at the cost of a slightly increased density in front of the obstacle. The rates are $\alpha_1 = 0.2$, $\alpha_2 = 0.4$, $\beta_1 = 0.4$ and $\beta_2 = 0.2$.}
\label{fig:2dhole}
\end{figure}

\section{Summary \& Outlook}
In this paper we analyzed a model for crowded transport with a single active species. We started by giving some details about the modelling then proceeding with two existence proofs in the stationary case. Next we analysed the flow characteristic of our model in the case of small diffusion. In one space dimension, we were able to recover three different phases that were already observed in the stochastic model \cite{Wood2009}. Further investigation showed however that the continous model can produce fluxes that exceed the value $j=1/4$ which do not occur on the discrete level. We concluded by presenting some numerical examples in one and two spatial dimensions.\\
Our analysis and especially the numerical examples in two spatial dimensions suggest that interesting phenomena can occur when dealing with more than one active species. Since each species has its own in- and outflow rate, it is not clear whether one would again observe different phases, clearly separated by certain values of these parameters. Also, to prove existence in the case $M>1$ becomes much more involved. Regarding the numerical discretization, a scheme that uses the reformulated problem in entropy variables might be an alternative to the direct approach used here.

\section*{Acknowledgements}
MB acknowledges support by ERC via Grant EU FP 7 - ERC Consolidator Grant 615216 LifeInverse. The work of JFP was supported by DFG via Grant 1073/1-1, by the Daimler and Benz Stiftung via Post-Doc Stipend 32-09/12 and by the German Academic Exchange Service via PPP grant No. 56052884. The authors would like to thank M.-T. Wolfram for suggesting the numerical example with obstacle.

\bibliographystyle{plain}
\bibliography{flow}

\section*{Appendix}
In this appendix we will detail the calculation of explicit solutions to the one-dimensional equation \eqref{uniflow1deq}. Since the flux $j$ is constant in 1D, we can integrate this equation to obtain the ordinary differential equation
\begin{align}\label{eq:integrated}
 -\partial_x\rho + \rho(1-\rho) = j
\end{align}
supplemented with the boundary conditions $j= \alpha(1-\rho(0))$ and $j=\beta \rho(1)$. We shall separately discuss the cases of constant density, maximal flux and the general case $j < 1/4$:
\begin{enumerate}
 \item $\rho = const$: Working with \eqref{eq:integrated}, the problem reduces to an overdetermined algebraic system of equations, namely
 \begin{align*}
  j &= \alpha ( 1 - \rho ),\\
  j &= \beta \rho,\\
  j &= \rho(1-\rho).
 \end{align*}
This is solvable if and only if $\alpha + \beta = 1$ and we obtain $\rho = \alpha = 1-\beta$ and $j=\alpha(1-\alpha)=\beta(1-\beta)$. In particular, the maximal flux $j = 1/4$ is achieved for $\alpha=\beta=1/2$, only.
 
\item $j = 1/4$: Note that \eqref{eq:integrated} with $j=1/4$ is also known as ``logistic equation with harvesting'' in the context of population dynamics, cf. \cite{Brauer1975,Cooke1986}.  In this case, \eqref{eq:integrated} becomes 
\begin{align*}
-\eps\partial_x \rho - \left(\rho - \frac{1}{2}\right)^2 = 0,
\end{align*}
which yields
\begin{align}\label{eq:sol_explicit_14}
\rho(x) = \frac{1}{2} + \frac{\eps}{x+c},
\end{align}
with the constant $c$ to be determined by the boundary conditions. We have
\begin{align*}
\frac{1}{4} &= \alpha(1-\rho(0)) = \alpha\left(\frac12 - \frac{\eps}{c_1}\right)\quad\Rightarrow c_1 = \frac{4\alpha\epsilon}{2\alpha-1},\\
\frac{1}{4} &= \beta\rho(1) = \beta\left(\frac12 + \frac{\eps}{1+c_2}\right)\quad\Rightarrow c_2 = \frac{4\beta\epsilon}{2\beta-1}-1.
\end{align*}
In order to obtain a single, continous solution we have to ensure the two conditions
\begin{align*}
 c_1 = c_2\text{ and either } c_1 = c_2 > 0 \text{ or } c_1 = c_2 < -1.
\end{align*}
Very elementary but rather tedious calculations show that this amounts to the following two conditions on $\alpha,\,\beta$ and $\eps$
\begin{align*}
 \frac{1}{2}\frac{1+2\eps}{4\eps+1} < \alpha &< \frac{1}{2}\quad\text{ and }\quad \beta = \frac{1}{2}\frac{4\alpha\eps+2\alpha-1}{8\alpha\eps+2\alpha-2\eps-1}\quad\text{for }c>0,\\
 \frac{1}{2}\frac{1+2\eps}{4\eps+1} < \beta &< \frac{1}{2}\quad\text{ and }\quad \alpha = \frac{1}{2}\frac{4\beta\eps+2\alpha-1}{8\beta\eps+2\beta-2\eps-1}\quad\text{for }c<-1,
\end{align*}
with $c:=c_1=c_2$.
As already discussed in section \ref{sec:explicit}, this in interesting since solution with maximal flux occur for values of $\alpha$, $\beta < 1/2$ which is in contrast to the discrete model, \cite{Wood2009}. See also \ref{sec:num1d} for some numerical examples that confirm this observation.
\item $j \neq 1/4$: In this case the explicit solutions is given by 
\begin{align}
\rho(x)  = \frac{1}{2} - \frac{\sqrt{4j-1}}{2}\tan\left(\frac{1}{2}\frac{\sqrt{4j-1}(x+c)}{\eps}\right).
\end{align}
The boundary conditions reduce to the following nonlinear algebraic system
\begin{align*}
j &= \alpha \left(\frac12+\frac{\sqrt{4j-1}}{2}\tan\left(\frac{1}{2}\frac{\sqrt{4j-1}(1+c)}{\eps}\right)\right),\\
j &= \beta\left(\frac12 - \frac{\sqrt{4j-1}}{2}\tan\left(\frac{1}{2}\frac{\sqrt{4j-1}c}{\eps}\right)\right).
\end{align*}
Solving these two conditions for the constant $c$, we finally end up with a non-linear equation determining $j$ given by
\begin{align*}
 \arctan\left(\frac{2j-\alpha}{\alpha\sqrt{4j-1}}\right) = \frac{\sqrt{4j-1}}{2\eps} + \arctan\left(\frac{2j-\beta}{\beta\sqrt{4j-1}}\right).
\end{align*}
This equation can be solved for example by applying Newton's method.
\end{enumerate}
Note that using these calculation, one can basically calculate the solution to \eqref{uniflow1deq} (with $u = 1$) explicitly for $j=1/4$ or $\rho=const$ and for other values of $j$ by solving a system of non-linear equations.



\end{document}